\documentclass[onefignum,onetabnum]{siamart251216}

\usepackage{amsfonts}
\usepackage{amssymb}
\usepackage{bm}
\usepackage{booktabs}
\usepackage{graphicx}
\usepackage{mathtools}
\usepackage{subcaption}

\usepackage[T1]{fontenc}
\usepackage{lmodern,microtype}

\graphicspath{{./}}
\DeclareGraphicsExtensions{.pdf,.png,.jpg}

\newcommand{\etal}{\emph{et al.}}
\newcommand{\ie}{i.e.,~}

\newcommand{\HydroTrixi}{\texttt{HydroTrixi.jl}}
\newcommand{\Trixi}{\texttt{Trixi.jl}}

\newcommand{\dd}{\mathop{}\!\mathrm{d}}
\newcommand{\vc}[1]{\bm{#1}}
\newcommand{\mat}[1]{\bm{\mathit{#1}}}
\newcommand{\T}{\mathrm{T}}
\renewcommand{\P}{\mathbb{P}}

\newcommand{\watercontent}{\vartheta}
\newcommand{\pressurehead}{\psi}
\newcommand{\conductivity}{\kappa}
\newcommand{\capacity}{c}

\newsiamremark{remark}{Remark}
\AddToHook{env/remark/begin}{\crefalias{theorem}{remark}}
\newsiamthm{assumption}{Assumption}

\crefname{equation}{}{}

\headers{Mass-conservative adaptive Richards solver}
{T. Montoya and R. J. Spiteri}

\title{A Fully Discrete, Mass-Conservative Adaptive Discontinuous
  Spectral-Element Method for the One-Dimensional Richards Equation in Mixed Form}%

\author{Tristan Montoya\thanks{Department of Computer Science,
  University of Saskatchewan, Saskatoon, SK, Canada
  (\email{tristan.montoya@usask.ca}, \email{spiteri@cs.usask.ca}).}
\and Raymond J. Spiteri\footnotemark[1]}

\ifpdf
\hypersetup{
  pdftitle={A Fully Discrete, Mass-Conservative Adaptive Discontinuous
  Spectral-Element Method for the One-Dimensional Richards Equation in Mixed Form},
  pdfauthor={T. Montoya and R. J. Spiteri}
}
\fi

\begin{document}
\maketitle

\begin{abstract}
  We present and analyze a spatially and temporally adaptive high-order method for the one-dimensional Richards equation describing vertical water transport in unsaturated soil, with an emphasis on resolving sharp wetting fronts while satisfying a fully discrete water mass balance.
  The spatial discretization uses an arbitrary-order local discontinuous Galerkin spectral-element formulation with collocated Legendre--Gauss--Lobatto quadrature, and the time integration is based on a high-order Rosenbrock--Wanner method.
  We focus primarily on the mixed formulation of the Richards equation, which advances water content as the conserved variable while evaluating the nonlinear diffusion operator in terms of pressure head; the resulting fully discrete scheme satisfies the water mass balance in exact arithmetic, independent of spatial resolution and time step size, including under spatial and temporal adaptivity.
  In contrast, the pressure-head formulation advances pressure head through a nonlinear capacity function and satisfies the corresponding conservation identity only at the semi-discrete level.
  Numerical experiments examine convergence through spatial and temporal refinement studies, with mass conservation assessed under spatial and temporal adaptivity using the recently proposed solver flux output method (SFOM).
  The SFOM results confirm machine-precision mass conservation for the mixed formulation, whereas the pressure-head formulation exhibits mass balance errors above roundoff.
  Moreover, in the temporal refinement studies, the mixed formulation yields more accurate solutions than the pressure-head formulation for a prescribed time step size.
  The combination of high-order accuracy, spatial and temporal adaptivity, and fully discrete mass conservation makes the proposed mixed-form discontinuous spectral-element method a promising approach for simulating unsaturated flow in next-generation land models.
\end{abstract}



\section{Introduction}\label{sec:introduction}
The one-dimensional Richards equation \cite{richardson_weather_prediction_1922,richards_capillary_1931} describes the vertical transport of water in variably saturated soil under gravitational and capillary forces.
It models conservation of water mass by relating the temporal change in water content to the spatial variation of the hydraulic flux, which depends nonlinearly on the pressure head and its gradient through the Darcy--Buckingham law.
Despite its apparent simplicity, the Richards equation poses significant challenges for numerical methods due to the strong nonlinearity of the flux and storage terms as well as its transition from parabolic behaviour to an elliptic limit as the soil nears saturation.
These features give rise to a wide range of spatial and temporal scales, particularly as sharp wetting fronts develop and propagate, making accurate resolution computationally expensive with fixed grids and constant time step sizes.
Furthermore, the semi-discrete systems obtained from discretizing the Richards equation in space typically result in stiff systems of ordinary differential equations (ODEs) or differential-algebraic equations (DAEs), which require careful treatment to maintain stability and accuracy without resorting to prohibitively small time steps.
The review papers by Farthing and Ogden~\cite{farthing_richards_review_2017}, Zha \etal~\cite{zha_review_numerical_solution_2019}, and Chatzikamaris \etal~\cite{chatzikamaris_numerical_solvers_2026} provide comprehensive discussions of the challenges encountered in the numerical solution of the Richards equation.
\par
Numerical methods that combine adaptive spatial and temporal discretization strategies can significantly reduce the computational cost of subsurface flow simulations while maintaining the accuracy needed to resolve sharp wetting fronts and other localized phenomena (see, for example, Miller \etal~\cite{miller_spatially_temporally_adaptive_2006}).
Relative to the second-order spatial discretization schemes commonly used in hydrologic modelling, high-order methods can achieve more accurate solutions for a given number of degrees of freedom in smooth solution regimes, with discontinuous Galerkin methods showing particular promise in the context of the Richards equation due to their local conservation properties and compact stencils (see, for example, Li \etal~\cite{li_local_ldg_2007,li_adaptive_ldg_2007}, Dolejší \etal~\cite{dolejsi_adaptive_space_time_2019,dolejsi_error_estimates_adaptivity_2024}, Clément \etal~\cite{clement_adaptive_dg_2021}, and Tian \etal~\cite{tian_h_adaptive_ldg_2027}).
However, ensuring an exact fully discrete water mass balance that is independent of the time step size and preserved through adaptive mesh refinement and coarsening remains a significant challenge, even for spatial discretizations that are locally conservative at the semi-discrete level.
This work addresses these challenges by developing and analyzing a spatially and temporally adaptive high-order method that ensures exact fully discrete mass conservation over the complete simulation, including both implicit time integration and mesh adaptation.
The method combines the following key components:
\begin{itemize}
  \item An arbitrary-order spectral-element approximation with collocated Legendre--Gauss--Lobatto (LGL) quadrature (see, for example, Kopriva~\cite{kopriva_implementing_spectral_methods_2009}) is used together with a local discontinuous Galerkin (LDG) formulation~\cite{cockburn_ldg_1998}, resulting in a compact spatial stencil and a diagonal element mass matrix.
  \item The mixed formulation of the Richards equation is used, in which water content is advanced as the conserved variable while the pressure head is retained within the spatial operator to evaluate the nonlinear flux.
  For comparison, we also consider the pressure-head formulation, in which pressure head is advanced through the nonlinear capacity function using the same spatial discretization; a purely water-content-based formulation is left for future work.
  \item Tree-based mesh adaptivity implemented in \Trixi{}~\cite{ranocha_trixi_2022} is used to concentrate resolution around sharp wetting fronts, where we employ a refinement indicator based on the total variation of the water content within each element together with a mass-conservative solution transfer procedure.
  \item A linearly implicit Rosenbrock--Wanner method~\cite{hairer_solving_1996} is used for time integration of the resulting stiff system of DAEs, providing high-order accuracy while avoiding iterative nonlinear solves and implementing adaptive time step control based on an embedded error estimate.
  \item The fully discrete conservation analysis uses the solver flux output method (SFOM) introduced by Ireson \etal~\cite{ireson_openre_2023} to track cumulative boundary fluxes with the same time integration method used to advance the solution.
  Under standard stage solvability assumptions, the resulting augmented DAE system preserves a linear invariant that guarantees an exact water mass balance in exact arithmetic for every accepted time step, which is preserved across mesh updates.
  \item Numerical experiments demonstrate high-order spatial and temporal convergence for a manufactured solution as well as mass conservation to roundoff under spatial and temporal adaptivity for canonical infiltration problems.
\end{itemize}
To our knowledge, the method presented in this work appears to be the first spatially and temporally adaptive high-order discretization for the Richards equation for which an exact fully discrete water mass balance is established, including time integration as well as mesh refinement and coarsening.
\par
The remainder of this paper is organized as follows.
\Cref{sec:mathematical_formulation} presents the mixed and pressure-head formulations of the Richards equation. \Cref{sec:spatial_discretization} describes the spatial discretization, including the spectral-element approximation, LDG formulation, spatial adaptivity, and mass-conservative solution transfer. \Cref{sec:temporal_discretization} introduces the adaptive Rosenbrock--Wanner time integration and establishes the fully discrete water mass balance. \Cref{sec:numerical_experiments} presents numerical experiments using an open-source Julia implementation of the proposed approach, in which we assess spatial and temporal convergence as well as mass conservation. \Cref{sec:conclusions} summarizes the main findings and discusses future work.

\section{Mathematical formulation}\label{sec:mathematical_formulation}

We consider vertical flow in a one-dimensional soil column of length $L > 0$, with depth measured downward from the soil surface.
Denoting the pressure head at depth $z \in [0,L]$ and time $t \geq 0$ by $\pressurehead(z,t)$, the Richards equation is defined as
\begin{equation}\label{eq:richards_equation}
  \partial_t \watercontent(\pressurehead(z,t)) =
  \partial_z f(\pressurehead(z,t),\partial_z\pressurehead(z,t)),
  \quad z \in [0,L],\quad t > 0,
\end{equation}
where the nonlinear flux function is defined by $f(\pressurehead,\partial_z\pressurehead) \coloneqq \conductivity(\pressurehead)(\partial_z \pressurehead - 1)$, with $\watercontent$ and $\conductivity$ denoting the water content and hydraulic conductivity functions provided by a constitutive model.
The unit gravitational-gradient contribution arises from the derivative of the total hydraulic head $\psi - z$, neglecting velocity head, under the convention that $z$ is positive downward.
In this work, we restrict our attention to pressure-head values in the unsaturated regime $\pressurehead < 0$.
Over this range, we assume that $\watercontent$ is twice continuously differentiable with a positive first derivative and that $\conductivity$ is continuously differentiable and non-negative.
\par
We supplement \cref{eq:richards_equation} with an initial pressure-head profile $\pressurehead(z,0)$.
A Dirichlet boundary condition (i.e., a prescribed pressure head) or a Neumann boundary condition (i.e., a prescribed flux) may be imposed at the top or bottom of the soil column, and we assume that all boundary data are continuously differentiable in time.
Integrating \cref{eq:richards_equation} in space and time, we see that the change in total volumetric water storage per unit cross-sectional area (hereafter referred to as the total water storage) is equal to the time-integrated net flux of water through its boundaries:
\begin{equation}\label{eq:continuous_water_mass_balance}
  \begin{aligned}
    &\int_0^L \watercontent(\pressurehead(z,t))\dd z
    -
    \int_0^L \watercontent(\pressurehead(z,0))\dd z \\
    &\quad =
    \int_0^t f(\pressurehead(L,s),\partial_z\pressurehead(L,s))\dd s
    -
    \int_0^t f(\pressurehead(0,s),\partial_z\pressurehead(0,s))\dd s.
  \end{aligned}
\end{equation}
For constant water density and column cross-sectional area, the total water storage is proportional to the total water mass; accordingly, we refer to this balance as water mass conservation.

\subsection{Mixed formulation}\label{sec:mixed_formulation}

As written in \cref{eq:richards_equation}, the conservative variable is the water content $\watercontent(\pressurehead(z,t))$, not the pressure head $\pressurehead(z,t)$.
Because water content depends nonlinearly on pressure head, a time integration method that advances the pressure head and preserves only linear invariants of its state variables does not automatically conserve mass.
Celia \etal~\cite{celia_mass_conservative_richards_1990} sought to address this issue by adopting a mixed formulation of the Richards equation, in which the water content is advanced in time as a conserved variable while the flux is evaluated in terms of the pressure head.
To facilitate time integration using general-purpose DAE solvers, we retain water content and pressure head as separate solution variables, similarly to Kees and Miller~\cite{kees_higher_order_time_integration_2002}, Cumming \etal~\cite{cumming_mass_conservative_2011}, and Bootsma \etal~\cite{bootsma_modified_picard_2026}.
Denoting the conserved water-content variable by $\theta(z,t)$, we express the mixed-form Richards equation as
\begin{subequations}\label{eq:richards_mixed}
  \begin{align}
    \partial_t \theta(z,t) &=
    \partial_z f(\pressurehead(z,t),\partial_z\pressurehead(z,t)),
    \label{eq:richards_mixed_balance} \\
    \theta(z,t) &= \watercontent(\pressurehead(z,t)),
    \label{eq:richards_mixed_constitutive}
  \end{align}
\end{subequations}
where \cref{eq:richards_mixed_constitutive} is a pointwise algebraic constraint that relates the water content to the pressure head through the constitutive model.

\subsection{Pressure-head formulation}\label{sec:pressure_head_formulation}

For comparison, we also consider the pressure-head formulation of the Richards equation, which results from applying the chain rule to the time derivative in \cref{eq:richards_equation} to obtain
\begin{equation}\label{eq:richards_pressure_head}
  \capacity(\pressurehead(z,t))\partial_t \pressurehead(z,t) =
  \partial_z f(\pressurehead(z,t),\partial_z\pressurehead(z,t)),
\end{equation}
where the capacity function is given by $\capacity(\pressurehead) \coloneqq \watercontent'(\pressurehead)$.
Because $\capacity(\pressurehead)>0$ over the range of pressure-head values considered, we can divide \cref{eq:richards_pressure_head} by $\capacity(\pressurehead(z,t))$ to solve explicitly for the time derivative of the pressure head.
Although the mixed and pressure-head formulations are equivalent at the continuous level for smooth solutions, the distinction between evolving $\theta(z,t)$ as a conserved quantity in \cref{eq:richards_mixed} and evolving $\pressurehead(z,t)$ according to \cref{eq:richards_pressure_head} is central to the fully discrete mass conservation result discussed in \cref{sec:fully_discrete_mass_conservation}.

\section{Spatial discretization}\label{sec:spatial_discretization}

In this section, we describe the LDG spectral-element discretization used for both the mixed and pressure-head formulations.

\subsection{Spectral-element approximation}\label{sec:spectral_element_approximation}

Following a standard nodal spectral-element construction (as described, for example, in \cite{kopriva_implementing_spectral_methods_2009}), let the column $[0,L]$ be partitioned into $K$ non-overlapping elements $[z_{k-1},z_k]$ with $0 = z_0 < \cdots < z_K = L$.
With the element size defined as $h_k \coloneqq z_k - z_{k-1}$, each element can be mapped from the reference interval $[-1,1]$ through the affine transformation
\begin{equation}\label{eq:element_map}
  X_k(\xi) \coloneqq \frac{h_k}{2}\xi + \frac{z_k+z_{k-1}}{2},
\end{equation}
which has a constant Jacobian given by $J_k = h_k/2$.
Under the mapping $z = X_k(\xi)$, we approximate the pressure head and its gradient on each element as
\begin{subequations}\label{eq:polynomial_expansion}
  \begin{align}
    \pressurehead(X_k(\xi),t) &\approx
    \pressurehead_k^N(\xi,t) \coloneqq \sum_{i=0}^N \pressurehead_{k,i}(t)\ell_i(\xi), \\
    \frac{\partial \pressurehead}{\partial z}(X_k(\xi),t) &\approx
    g_k^N(\xi,t) \coloneqq \sum_{i=0}^N g_{k,i}(t)\ell_i(\xi),
  \end{align}
\end{subequations}
where $N \geq 1$ is the chosen polynomial degree of the spatial approximation, and the Lagrange basis functions $\{\ell_i\}_{i=0}^N$ for the space $\P_N([-1,1])$ of real-valued polynomials of degree at most $N$ are defined for $N+1$ distinct interpolation nodes $\{\xi_i\}_{i=0}^N$ as
\begin{equation}\label{eq:lagrange_polynomials}
  \ell_i(\xi) \coloneqq
  \prod_{\substack{m=0 \\ m\neq i}}^N
  \frac{\xi-\xi_m}{\xi_i-\xi_m},
  \quad i \in \{0,\ldots,N\}.
\end{equation}
In this work, we choose $\{\xi_i\}_{i=0}^N$ to be the LGL nodes, which are the roots of $(1-\xi^2)P_N'(\xi)$, where $P_N$ is the Legendre polynomial of degree $N$.
The corresponding positive quadrature weights $\{\omega_i\}_{i=0}^N$ are then given by
\begin{equation}\label{eq:lgl_weights}
  \omega_i \coloneqq \int_{-1}^1 \ell_i(\xi) \dd \xi =
  \frac{2}{N(N+1)}\frac{1}{P_N(\xi_i)^2},
  \quad i \in \{0,\ldots,N\},
\end{equation}
resulting in a quadrature rule that is exact for all polynomials of degree at most $2N-1$.
Adopting a zero-based indexing convention, we construct the collocation derivative matrix $\mat{D}$ and diagonal weight matrix $\mat{W}$ associated with the LGL quadrature rule as
\begin{equation}\label{eq:sbp_matrices_1d}
  D_{ij} \coloneqq \frac{\dd \ell_j}{\dd \xi}(\xi_i),
  \quad
  W_{ij} \coloneqq \omega_i\delta_{ij}, \quad i,j \in \{0,\ldots,N\}.
\end{equation}
Because the approximation space contains the constant function, the derivative matrix satisfies $\mat{D}\vc{1}=\vc{0}$, where $\vc{1} \coloneqq (1,\ldots,1)^\T$ is a vector of ones and $\vc{0} \coloneqq (0,\ldots,0)^\T$ is a vector of zeros, both of length $N+1$.
Furthermore, as shown by Gassner~\cite{gassner_skew_symmetric_dg_2013}, the derivative and quadrature weight matrices satisfy the summation-by-parts property
\begin{equation}\label{eq:sbp_property_1d}
  \mat{W}\mat{D} + \mat{D}^\T\mat{W} = \vc{e}_N\vc{e}_N^\T - \vc{e}_0\vc{e}_0^\T,
\end{equation}
where we define the vectors $\vc{e}_0 \coloneqq (1,0,\ldots,0)^\T$ and $\vc{e}_N \coloneqq (0,\ldots,0,1)^\T$.

\subsection{Local discontinuous Galerkin method}\label{sec:ldg_formulation}

We now construct the LDG discretization of the spatial terms shared by the mixed and pressure-head formulations, leaving the choice of evolved unknown to \cref{sec:semidiscrete_formulation}.
Following Cockburn and Shu~\cite{cockburn_ldg_1998}, we first rewrite the second-order problem as a first-order system given by
\begin{subequations}\label{eq:first_order_system_physical}
  \begin{align}
    g(z,t) &= \partial_z\pressurehead(z,t),
    \label{eq:first_order_gradient_physical} \\
    \partial_t\watercontent(\pressurehead(z,t)) &= \partial_z
    f(\pressurehead(z,t),g(z,t)). \label{eq:first_order_balance_physical}
  \end{align}
\end{subequations}
Expressing the system in reference coordinates using \cref{eq:element_map}, inserting the polynomial approximations in \cref{eq:polynomial_expansion}, and integrating by parts against a test function then gives the following weak formulation in space: find $\pressurehead_k^N(\cdot,t),g_k^N(\cdot,t) \in \P_N([-1,1])$ satisfying
\begin{subequations}\label{eq:ldg_weak_form}
  \begin{equation}\label{eq:ldg_gradient}
    \begin{aligned}
      \int_{-1}^{1} J_k w(\xi) g_k^N(\xi,t) \dd \xi
      &= -\int_{-1}^{1} \partial_\xi
      w(\xi)\pressurehead_k^N(\xi,t)\dd \xi \\
      &\quad
      -w(-1)\pressurehead_{k-1}^\star(t) + w(1)\pressurehead_k^\star(t),
    \end{aligned}
  \end{equation}
  \begin{equation}\label{eq:ldg_balance}
    \begin{aligned}
      \int_{-1}^{1}
      J_k v(\xi) \partial_t\watercontent(\pressurehead_k^N(\xi,t)) \dd \xi
      &= -\int_{-1}^{1}
      \partial_\xi v(\xi) f(\pressurehead_k^N(\xi,t),g_k^N(\xi,t)) \dd \xi \\
      &\quad -v(-1)f_{k-1}^\star(t) + v(1)f_k^\star(t),
    \end{aligned}
  \end{equation}
\end{subequations}
for all test functions $v,w \in \P_N([-1,1])$.
The quantities $\pressurehead_{k-1}^\star(t)$ and $\pressurehead_k^\star(t)$ are unique numerical traces of the pressure head at the upper endpoint $z_{k-1}$ and the lower endpoint $z_k$, respectively, of element $k$, and $f_{k-1}^\star(t)$ and $f_k^\star(t)$ are the corresponding numerical fluxes.
Using a minimally dissipative LDG interface treatment with alternating traces and no interior penalty (see, for example, Cockburn and Dong~\cite{cockburn_minimal_dissipation_2007}), we take the numerical pressure-head trace from the upper side of each interior interface as
\begin{equation}\label{eq:ldg_traces}
  \pressurehead_k^\star(t) \coloneqq \pressurehead_k^N(1,t),
  \quad k \in \{1,\ldots,K-1\},
\end{equation}
and the numerical flux from the lower side as
\begin{equation}\label{eq:ldg_numerical_fluxes}
  f_k^\star(t) \coloneqq f(\pressurehead_{k+1}^N(-1,t),g_{k+1}^N(-1,t)),
  \quad k \in \{1,\ldots,K-1\}.
\end{equation}
The numerical traces and fluxes at the top and bottom boundaries are determined by the chosen boundary conditions, which we assume to be of Dirichlet or Neumann type.
For instance, a Dirichlet boundary condition of the form $\pressurehead(0,t) = \pressurehead_{\mathrm T}(t)$ can be weakly imposed at the top boundary as
\begin{subequations}
  \begin{align}
    \pressurehead_0^\star(t) & \coloneqq \pressurehead_{\mathrm T}(t), \\
    f_0^\star(t)             & \coloneqq f(\pressurehead_1^N(-1,t),g_1^N(-1,t)) +
    \tau(\pressurehead_{\mathrm T}(t),h_1,N)
    \left(\pressurehead_1^N(-1,t)-\pressurehead_{\mathrm T}(t)\right),
  \end{align}
\end{subequations}
where $\tau(\pressurehead, h, N) \coloneqq N(N+1)\conductivity(\pressurehead)/h$ is a penalty parameter that provides a standard scaling of the pressure-head jump with respect to the polynomial degree $N$ and element size $h$ under collocated LGL quadrature (see, for example, Manzanero \etal~\cite[Appendix~A]{manzanero_bassi_rebay_2018}).
At the bottom boundary, we can similarly impose $\pressurehead(L,t) = \pressurehead_{\mathrm B}(t)$ as
\begin{subequations}
  \begin{align}
    \pressurehead_K^\star(t) & \coloneqq \pressurehead_{\mathrm B}(t), \\
    f_K^\star(t)             & \coloneqq f(\pressurehead_K^N(1,t),g_K^N(1,t)) -
    \tau(\pressurehead_{\mathrm B}(t),h_K,N)
    \left(\pressurehead_K^N(1,t)-\pressurehead_{\mathrm B}(t)\right),
  \end{align}
\end{subequations}
where we note that the penalty term is subtracted to account for the reversed orientation of the outward normal.
Neumann boundary conditions of the form $f(\pressurehead(0,t),g(0,t)) = f_{\mathrm{T}}(t)$ or $f(\pressurehead(L,t),g(L,t)) = f_{\mathrm{B}}(t)$ are imposed by setting the numerical flux to the prescribed value and taking the numerical pressure-head trace from the adjacent interior state.

\subsection{Semi-discrete formulation}\label{sec:semidiscrete_formulation}

To obtain a system of ODEs or DAEs, we take the test functions in \cref{eq:ldg_weak_form} to be the Lagrange basis functions and approximate the integrals using the LGL quadrature rule.
Evaluating the approximate pressure head and its gradient on element $k$ at the LGL quadrature nodes as $\vc{\pressurehead}_k(t) \coloneqq (\pressurehead_{k,0}(t),\ldots,\pressurehead_{k,N}(t))^\T$ and $\vc{g}_k(t) \coloneqq (g_{k,0}(t),\ldots,g_{k,N}(t))^\T$, respectively, we compute the nonlinearities pointwise as
\begin{subequations}\label{eq:pointwise_nonlinearities}
  \begin{align}
    \vc{f}_k(t) &\coloneqq
    \left(
    f(\pressurehead_{k,0}(t),g_{k,0}(t)),\ldots,
    f(\pressurehead_{k,N}(t),g_{k,N}(t))
    \right)^\T,
    \label{eq:pointwise_flux_vector} \\
    \mat{C}_k(t) &\coloneqq
    \operatorname{diag}\left(
    \capacity(\pressurehead_{k,0}(t)),\ldots,
    \capacity(\pressurehead_{k,N}(t))
    \right).
    \label{eq:pointwise_capacity_matrix}
  \end{align}
\end{subequations}
The remaining step is to choose the evolved unknown.
The semi-discrete mixed formulation approximates the water content on element $k$ as
\begin{equation}\label{eq:water_content_approximation}
  \theta(X_k(\xi),t) \approx \theta_k^N(\xi,t)
  \coloneqq \sum_{i=0}^N \theta_{k,i}(t)\ell_i(\xi),
\end{equation}
and evolves the vector $\vc{\theta}_k(t) \coloneqq (\theta_{k,0}(t),\ldots,\theta_{k,N}(t))^\T$ of water-content values at the LGL quadrature nodes, resulting in a semi-discrete system given by
\begin{subequations}\label{eq:semidiscrete_mixed}
  \begin{align}
    \vc{g}_k(t) &= \frac{1}{J_k}\left(
    -\mat{W}^{-1}\mat{D}^\T\mat{W}\vc{\pressurehead}_k(t)
    -\mat{W}^{-1}\vc{e}_0\pressurehead_{k-1}^\star(t)
    +\mat{W}^{-1}\vc{e}_N\pressurehead_k^\star(t)
    \right),
    \label{eq:semidiscrete_mixed_gradient} \\
    \dot{\vc{\theta}}_k(t) &= \frac{1}{J_k}\left(
    -\mat{W}^{-1}\mat{D}^\T\mat{W}\vc{f}_k(t)
    -\mat{W}^{-1}\vc{e}_0 f_{k-1}^\star(t)
    +\mat{W}^{-1}\vc{e}_N f_k^\star(t)
    \right),
    \label{eq:semidiscrete_mixed_balance}
  \end{align}
\end{subequations}
subject to the algebraic constraint $\vc{\theta}_k(t) = (\watercontent(\pressurehead_{k,0}(t)),\ldots, \watercontent(\pressurehead_{k,N}(t)))^\T$.
The semi-discrete pressure-head formulation likewise follows from evolving the vector $\vc{\pressurehead}_k(t)$ of pressure-head values through the capacity function as
\begin{subequations}\label{eq:semidiscrete_pressure_head}
  \begin{align}
    \vc{g}_k(t) &= \frac{1}{J_k}\left(
    -\mat{W}^{-1}\mat{D}^\T\mat{W}\vc{\pressurehead}_k(t)
    -\mat{W}^{-1}\vc{e}_0\pressurehead_{k-1}^\star(t)
    +\mat{W}^{-1}\vc{e}_N\pressurehead_k^\star(t)
    \right),
    \label{eq:semidiscrete_pressure_head_gradient} \\
    \mat{C}_k(t)\dot{\vc{\pressurehead}}_k(t) &=
    \frac{1}{J_k}\left(
    -\mat{W}^{-1}\mat{D}^\T\mat{W}\vc{f}_k(t)
    -\mat{W}^{-1}\vc{e}_0 f_{k-1}^\star(t)
    +\mat{W}^{-1}\vc{e}_N f_k^\star(t)
    \right).
    \label{eq:semidiscrete_pressure_head_balance}
  \end{align}
\end{subequations}
The formulations in \cref{eq:semidiscrete_mixed,eq:semidiscrete_pressure_head} therefore share the same LDG approximation of the nonlinear flux divergence, but they differ in the variable to which the time derivative is applied.
We then obtain the following semi-discrete counterpart of \cref{eq:continuous_water_mass_balance}.

\begin{theorem}\label{thm:semidiscrete_mass_conservation}
  If the numerical flux is single-valued at every interior interface, then the mixed and pressure-head formulations satisfy the semi-discrete water mass balance
  \begin{equation}\label{eq:semidiscrete_water_mass_balance}
    \frac{\dd}{\dd t}
    \sum_{k=1}^{K} J_k\vc{1}^\T\mat{W}\vc{\theta}_k(t) =
    f_K^\star(t) - f_0^\star(t),
  \end{equation}
  where $\vc{\theta}_k(t)$ denotes the nodal water-content vector.
  In the mixed formulation, this vector is evolved directly; in the pressure-head formulation, it is evaluated from the evolved pressure-head vector according to $\vc{\theta}_k(t) \coloneqq (\watercontent(\pressurehead_{k,0}(t)),\ldots, \watercontent(\pressurehead_{k,N}(t)))^\T$.
\end{theorem}

\begin{proof}
  For the mixed formulation, left-multiplication of \cref{eq:semidiscrete_mixed_balance} by $J_k\vc{1}^\T\mat{W}$ gives the local water mass balance
  \begin{equation}\label{eq:local_water_mass_balance}
    \begin{aligned}
      J_k\vc{1}^\T\mat{W}\dot{\vc{\theta}}_k(t)
      &=
      -\vc{1}^\T\mat{D}^\T\mat{W}\vc{f}_k(t)
      -\vc{1}^\T\vc{e}_0 f_{k-1}^\star(t)
      +\vc{1}^\T\vc{e}_N f_k^\star(t) \\
      &= -f_{k-1}^\star(t) + f_k^\star(t),
    \end{aligned}
  \end{equation}
  where the second equality follows from the identities $\mat{D}\vc{1}=\vc{0}$ and $\vc{1}^\T\vc{e}_0=\vc{1}^\T\vc{e}_N=1$.
  For the pressure-head formulation, applying the chain rule in time to the nodal definition of $\vc{\theta}_k(t)$ gives $\dot{\vc{\theta}}_k(t)=\mat{C}_k(t)\dot{\vc{\pressurehead}}_k(t)$, so \cref{eq:semidiscrete_pressure_head_balance} yields the same elementwise water mass balance obtained in \cref{eq:local_water_mass_balance}.
  For either formulation, summing this balance over all elements and moving the time derivative outside the sum yields
  \begin{equation}\label{eq:telescoping_sum}
    \frac{\dd}{\dd t}
    \sum_{k=1}^{K} J_k\vc{1}^\T\mat{W}\vc{\theta}_k(t) =
    \sum_{k=1}^{K}\left( f_k^\star(t) - f_{k-1}^\star(t) \right),
  \end{equation}
  and hence the global semi-discrete balance in \cref{eq:semidiscrete_water_mass_balance} follows from the fact that the single-valued numerical fluxes cancel at the interior interfaces.
\end{proof}

\subsection{Spatial adaptivity}\label{sec:adaptive_mesh_refinement}

We use the hierarchical tree-based mesh structure and adaptivity mechanics implemented in \Trixi{}~\cite{ranocha_trixi_2022}, although the indicator and conservative transfer procedure employed in this work are specifically tailored to the Richards equation.
An $h$-adaptive approach is used, wherein refinement bisects an element into two children, coarsening merges two adjacent siblings into their parent, and the polynomial degree remains fixed across all elements.
In order to track moving fronts in the unsaturated regime, the elementwise refinement indicator is computed as the discrete total variation of the water content within a given element, normalized by the global nodal water-content range. 
We compute the refinement indicator for each element as
\begin{equation}
  \eta_k(t) \coloneqq
  \frac{\sum_{i=0}^N \omega_i
    \left|\sum_{j=0}^N D_{ij}\vartheta(\psi_{k,j}(t))\right|}
  {\theta_{\max}(t)-\theta_{\min}(t)+\epsilon},
\end{equation}
where $\epsilon>0$ is a small regularization parameter that prevents division by zero, and we compute the minimum and maximum nodal water-content values over all elements as
\begin{equation}
\theta_{\min}(t) \coloneqq \min_{k=1}^K \min_{i=0}^N \vartheta(\psi_{k,i}(t)), \quad
\theta_{\max}(t) \coloneqq \max_{k=1}^K \max_{i=0}^N \vartheta(\psi_{k,i}(t)).
\end{equation}
A simple two-threshold rule is used to determine whether an element should be refined, coarsened, or left unchanged.
Let $l_k$ denote the current refinement level of element $k$, such that the size of the element is given by $h_k = 2^{-l_k}L$, and let $l_{\min}$ and $l_{\max}$ denote the prescribed minimum and maximum refinement levels, respectively, with $l_{\min} < l_{\max}$.
For a coarsening threshold $\eta_{\mathrm{c}}$ and a refinement threshold $\eta_{\mathrm{r}}$, with $\eta_{\mathrm{c}} < \eta_{\mathrm{r}}$, we mark element $k$ for refinement if $l_k<l_{\max}$ and $\eta_k(t)>\eta_{\mathrm{r}}$, and we mark it for coarsening if $l_k>l_{\min}$ and $\eta_k(t)\leq\eta_{\mathrm{c}}$.
We split each element that is marked for refinement, and we merge a pair of sibling elements only if both are marked for coarsening.
Additional refinements are performed to ensure that the difference in refinement levels between neighbouring elements is at most one, and coarsening is only performed if the resulting mesh would satisfy this same condition.

\subsection{Mass-conservative solution transfer}\label{sec:conservative_solution_transfer}

When an element is refined to produce two children, the parent polynomial on the reference element, which we denote for an arbitrary transfer variable as $\phi_k^N \in \P_N([-1,1])$, is split into two child polynomials given by
\begin{equation}\label{eq:adaptive_refinement_restriction}
  \phi_{k,\mathrm L}^N(\zeta) \coloneqq \phi_k^N\left(\frac{\zeta-1}{2}\right), \quad
  \phi_{k,\mathrm R}^N(\zeta) \coloneqq \phi_k^N\left(\frac{\zeta+1}{2}\right),
\end{equation}
which are evaluated at the reference LGL nodes to obtain the nodal degrees of freedom after the mesh update.
When two sibling elements are combined to coarsen the mesh, we compute the $L^2$ projection of the child transfer-variable polynomials onto the parent element.
Specifically, given the child polynomials $\phi_{k,\mathrm L}^N,\phi_{k,\mathrm R}^N \in \P_N([-1,1])$, we solve for a parent polynomial $\phi_k^N \in \P_N([-1,1])$ such that
\begin{equation}\label{eq:adaptive_coarsening_projection}
  \begin{aligned}
    \int_{-1}^1 \phi_k^N(\xi)v(\xi)\dd\xi
    & = \frac{1}{2}\int_{-1}^1 \phi_{k,\mathrm L}^N(\zeta)v\left(\frac{\zeta-1}{2}\right)\dd\zeta \\
    & \quad + \frac{1}{2}\int_{-1}^1 \phi_{k,\mathrm R}^N(\zeta)v\left(\frac{\zeta+1}{2}\right)\dd\zeta
  \end{aligned}
\end{equation}
holds for all $v \in \P_N([-1,1])$.
This projection is evaluated using an $(N+1)$-point Legendre--Gauss quadrature rule, which is exact for the polynomial products involved.
We now consider the case in which water content is the transfer variable, which was proposed in \cite{miller_spatially_temporally_adaptive_2006} as a means of conserving mass across mesh updates for an adaptive finite-difference spatial discretization.
The following result shows that every such update preserves the discrete total water mass for our discontinuous spectral-element method.

\begin{theorem}\label{thm:conservative_solution_transfer}
  Let the transfer variable be the water-content polynomial defined in \cref{eq:water_content_approximation}.
  For the mixed formulation, its nodal coefficients $\theta_{k,i}(t)$ are evolved directly according to \cref{eq:semidiscrete_mixed_balance}, whereas for the pressure-head formulation, the polynomial is obtained through Lagrange interpolation with nodal values $\theta_{k,i}(t) = \watercontent(\pressurehead_{k,i}(t))$.
  If the solution transfer procedures in \cref{eq:adaptive_refinement_restriction,eq:adaptive_coarsening_projection} are applied for a mesh update at time $t$ with $\phi_k^N(\xi) = \theta_k^N(\xi,t)$, $\phi_{k,\mathrm L}^N(\zeta) = \theta_{k,\mathrm L}^N(\zeta,t)$, and $\phi_{k,\mathrm R}^N(\zeta) = \theta_{k,\mathrm R}^N(\zeta,t)$, every adaptive mesh update then preserves the total water storage in \cref{eq:semidiscrete_water_mass_balance}.
\end{theorem}

\begin{proof}
  First, we consider a parent element that is refined into two child elements.
  Denoting the Jacobian of the parent-element mapping by $J_k$, noting that the Jacobians of the child-element mappings are given by $J_{k,\mathrm L} = J_{k,\mathrm R} = J_k/2$, and using the fact that the polynomials $\theta_k^N(\cdot,t), \theta_{k,\mathrm L}^N(\cdot,t), \theta_{k,\mathrm R}^N(\cdot,t) \in \P_N([-1,1])$ are all integrated exactly by the LGL quadrature rule, we obtain the total water storage within the two child elements after refinement as
  \begin{equation}
    \begin{aligned}
      J_{k,\mathrm L}\vc{1}^\T\mat{W}\vc{\theta}_{k,\mathrm L}(t) &+ J_{k,\mathrm R}\vc{1}^\T\mat{W}\vc{\theta}_{k,\mathrm R}(t) \\
      &= \frac{J_k}{2}\int_{-1}^1\theta_{k,\mathrm L}^N(\zeta,t)\dd\zeta
      + \frac{J_k}{2}\int_{-1}^1\theta_{k,\mathrm R}^N(\zeta,t)\dd\zeta \\
      &= \frac{J_k}{2}\int_{-1}^1
      \theta_k^N\left(\frac{\zeta-1}{2},t\right)\dd\zeta
      + \frac{J_k}{2}\int_{-1}^1
      \theta_k^N\left(\frac{\zeta+1}{2},t\right)\dd\zeta \\
      &= J_k\int_{-1}^0\theta_k^N(\xi,t)\dd\xi + J_k\int_0^1\theta_k^N(\xi,t)\dd\xi \\
      &= J_k\int_{-1}^1\theta_k^N(\xi,t)\dd\xi \\
      &= J_k\vc{1}^\T\mat{W}\vc{\theta}_k(t),
    \end{aligned}
  \end{equation}
  which is the total water storage within the parent element before refinement.
  Second, we consider a pair of sibling elements that are combined under coarsening to form a single parent element.
  Setting $v=1$ and taking the water content as the transfer variable in \cref{eq:adaptive_coarsening_projection} gives
  \begin{equation}
    \int_{-1}^1\theta_k^N(\xi,t)\dd\xi =
    \frac{1}{2}\int_{-1}^1\theta_{k,\mathrm L}^N(\zeta,t)\dd\zeta +
    \frac{1}{2}\int_{-1}^1\theta_{k,\mathrm R}^N(\zeta,t)\dd\zeta.
  \end{equation}
  Multiplying both sides by $J_k$ and using the exactness of the LGL quadrature for the polynomial integrands, we therefore obtain
  \begin{equation}
    J_k\vc{1}^\T\mat{W}\vc{\theta}_k(t) =
    J_{k,\mathrm L}\vc{1}^\T\mat{W}\vc{\theta}_{k,\mathrm L}(t)
    + J_{k,\mathrm R}\vc{1}^\T\mat{W}\vc{\theta}_{k,\mathrm R}(t),
  \end{equation}
  which shows that the water storage within the parent element obtained after coarsening is equal to the sum of the child-element storage contributions before coarsening.
  Because the contributions from elements that are neither refined nor coarsened remain unchanged by the mesh update, summing the contributions over all elements before and after the mesh update demonstrates that the total water mass is preserved.
\end{proof}

\begin{remark}
  Because $\vartheta(\psi)$ is strictly increasing for $\psi<0$, its inverse is well defined on the corresponding range of water-content values. 
  However, conservation alone does not ensure that the transferred nodal water-content values lie in the domain of the inverse constitutive relation.
  We assume that this admissibility condition holds after each mesh update and reconstruct the nodal pressure-head vector as $\vc{\psi}_k(t) = (\vartheta^{-1}(\theta_{k,0}(t)),\ldots,\vartheta^{-1}(\theta_{k,N}(t)))^\T$ for every element $k$ on the updated mesh.
\end{remark}

\section{Temporal discretization}\label{sec:temporal_discretization}

By concatenating the elementwise unknowns in \cref{eq:semidiscrete_mixed} or \cref{eq:semidiscrete_pressure_head} to obtain a global state vector $\vc{y}(t)$, we obtain a system of the form
\begin{equation}\label{eq:temporal_mass_matrix_system}
  \mat{M}\dot{\vc{y}}(t) = \vc{\mathcal{F}}(\vc{y}(t),t),
\end{equation}
where, on each fixed mesh between adaptive mesh updates, $\mat{M}$ is a constant mass matrix, which may be singular, and $\vc{\mathcal{F}}$ is a nonlinear function of the state and time.
A mesh update changes the dimensions of the state vector and mass matrix consistently with the updated mesh, which, strictly speaking, yields a new system of the form in \cref{eq:temporal_mass_matrix_system} that is initialized through the transfer procedure described in \cref{sec:conservative_solution_transfer}.
When defining mappings on the state space, we omit the time dependence of the state variables; thus, $\vc{y}$ denotes a generic state, whereas $\vc{y}(t)$ denotes a solution trajectory.
\par
The global state vector for the mixed formulation concatenates the global water-content and pressure-head vectors, which are defined, respectively, as
\begin{equation}\label{eq:global_solution_vectors}
  \vc{\Theta}(t) \coloneqq
  \begin{pmatrix}
    \vc{\theta}_1(t) \\
    \vdots \\
    \vc{\theta}_K(t)
  \end{pmatrix},
  \quad
  \vc{\Psi}(t) \coloneqq
  \begin{pmatrix}
    \vc{\pressurehead}_1(t) \\
    \vdots \\
    \vc{\pressurehead}_K(t)
  \end{pmatrix}.
\end{equation}
Letting $\mat{I}$ denote the identity matrix of size $K(N+1)$ and $\mat{0}$ denote the zero matrix of the same dimensions, the resulting DAE system is then given by \cref{eq:temporal_mass_matrix_system} with
\begin{equation}\label{eq:temporal_mixed_time_data}
  \vc{y} =
  \begin{pmatrix}
    \vc{\Theta} \\
    \vc{\Psi}
  \end{pmatrix},
  \quad
  \mat{M} =
  \begin{pmatrix}
    \mat{I} & \mat{0} \\
    \mat{0} & \mat{0}
  \end{pmatrix},
  \quad
  \vc{\mathcal{F}}(\vc{y},t) =
  \begin{pmatrix}
    \boldsymbol{\mathcal{R}}(\vc{\Psi},t) \\
    \vc{\Theta} - \vc{\vartheta}(\vc{\Psi})
  \end{pmatrix},
\end{equation}
where $\boldsymbol{\mathcal{R}}(\vc{\Psi},t)$ is the spatial residual obtained by computing the gradient approximation on each element using \cref{eq:semidiscrete_mixed_gradient} and using the gradient values to evaluate the elementwise residuals in \cref{eq:semidiscrete_mixed_balance}.
The vector $\vc{\vartheta}(\vc{\Psi})$ is obtained by applying the constitutive relation to each entry in the global pressure-head vector $\vc{\Psi}$, and we initialize the water-content vector as $\vc{\Theta}(0) = \vc{\vartheta}(\vc{\Psi}(0))$, so that the algebraic constraint is satisfied exactly at the beginning of the simulation. 
Because the Jacobian of the algebraic block of $\vc{\mathcal{F}}(\vc{y},t)$ with respect to the vector $\vc{\Psi}$ of algebraic variables is given by $-\mat{C}(\vc{\Psi})$, the DAE system is of index one under the assumption that all nodal capacity values are strictly positive.
\par
For the pressure-head formulation, the state contains only $\vc{\Psi}(t)$.
Under the assumptions in \cref{sec:mathematical_formulation}, all nodal capacity values are strictly positive, so dividing both sides of \cref{eq:semidiscrete_pressure_head_balance} by the nodal values of the capacity function and concatenating contributions from all elements results in an ODE system given by \cref{eq:temporal_mass_matrix_system} with
\begin{equation}\label{eq:temporal_pressure_head_time_data}
  \vc{y} = \vc{\Psi},
  \quad
  \mat{M} = \mat{I},
  \quad
  \vc{\mathcal{F}}(\vc{y},t) =
  \mat{C}(\vc{\Psi})^{-1}\boldsymbol{\mathcal{R}}(\vc{\Psi},t),
\end{equation}
where $\mat{C}(\vc{\Psi})$ is a diagonal matrix with its diagonal entries obtained by evaluating the capacity function for each entry of $\vc{\Psi}$, and $\boldsymbol{\mathcal{R}}(\vc{\Psi},t)$ is the same spatial residual vector as in the mixed formulation.
It is also possible to retain the capacity function on the left-hand side (see, for example, Tocci \etal~\cite{tocci_accurate_economical_1997}), although we do not pursue this approach because it introduces a state-dependent mass matrix.

\subsection{Rosenbrock--Wanner methods}\label{sec:rosenbrock_wanner_methods}

We integrate \cref{eq:temporal_mass_matrix_system} in time using an adaptive Rosenbrock--Wanner method, which is a linearly implicit high-order method for stiff systems of ODEs and index-one DAEs (see, for example, Hairer and Wanner~\cite[Sections~IV.7 and~VI.4]{hairer_solving_1996}).
An $s$-stage embedded Rosenbrock--Wanner method is characterized by the coefficients $\alpha_{ij}$, $\gamma_{ij}$, $b_i$, and $\hat{b}_i$, with
\begin{equation}
  \alpha_{ij}=0 \quad \text{for } j \geq i,
  \quad
  \gamma_{ij}=0 \quad \text{for } j > i,
  \quad
  \gamma_{ii}=\gamma\neq 0.
\end{equation}
Given the solution $\vc{y}^n$ at time $t^n$, the method computes a trial step (which may be accepted or rejected) to advance the solution to time $t^n + \Delta t^n$, where $\Delta t^n$ is the trial step size and $t^0 = 0$, as
\begin{equation}\label{eq:rosenbrock_wanner_update_general}
  \vc{y}^{n+1} = \vc{y}^n + \sum_{i=1}^s b_i\vc{k}_i^n,
  \quad
  \hat{\vc{y}}^{n+1} = \vc{y}^n +
  \sum_{i=1}^s \hat{b}_i\vc{k}_i^n,
\end{equation}
where $\vc{y}^{n+1}$ is the primary approximation and $\hat{\vc{y}}^{n+1}$ is an embedded approximation used for error estimation, both of which are defined on the mesh at time $t^n$.
At each stage $i \in \{1, \ldots, s\}$, the stage solution vector $\vc{k}_i^n$ is obtained by solving the linear system of stage equations
\begin{equation}\label{eq:rosenbrock_wanner_stages}
  \begin{aligned}
    \bigg(
    \mat{M}
    - \gamma\Delta t^n
    \frac{\partial \vc{\mathcal{F}}}{\partial \vc{y}}(\vc{y}^n,t^n)
    \bigg)\vc{k}_i^n
    =  \Delta t^n\vc{\mathcal{F}}
    \bigg(\vc{y}^n + \sum_{j=1}^{i-1}\alpha_{ij}\vc{k}_j^n,
    t^n + \Delta t^n\sum_{j=1}^{i-1}\alpha_{ij}\bigg) \\
    + \ \Delta t^n
    \frac{\partial \vc{\mathcal{F}}}{\partial \vc{y}}(\vc{y}^n,t^n)
    \sum_{j=1}^{i-1}\gamma_{ij}\vc{k}_j^n
    + (\Delta t^n)^2
    \sum_{j=1}^{i}\gamma_{ij}
    \frac{\partial \vc{\mathcal{F}}}{\partial t}(\vc{y}^n,t^n),
  \end{aligned}
\end{equation}
where the right-hand side $\vc{\mathcal{F}}$ is defined using the spatial discretization at time $t^n$, and we note that the Jacobian $\partial \vc{\mathcal{F}}/\partial \vc{y}$ and the time derivative $\partial \vc{\mathcal{F}}/\partial t$ are held fixed across all stages of each trial step.
Although the analysis in this paper is presented in terms of a general $s$-stage method, all numerical experiments use the eight-stage, fifth-order \texttt{Rodas5P} scheme~\cite{steinebach_rodas5p_2023} implemented within the \texttt{OrdinaryDiffEqRosenbrock.jl} package from the \texttt{DifferentialEquations.jl} ecosystem~\cite{rackauckas_differentialequations_2017}.

\subsection{Linear system and Jacobian evaluation}\label{sec:jacobian_evaluation}

The Jacobian matrix for the mixed formulation in \cref{eq:temporal_mixed_time_data} has the block structure
\begin{equation}\label{eq:mixed_residual_jacobian}
  \frac{\partial \vc{\mathcal{F}}}{\partial \vc{y}}(\vc{y}^n,t^n) =
  \begin{pmatrix}
    \mat{0} & (\partial \vc{\mathcal{R}}/\partial \vc{\Psi}) (\vc{\Psi}^n,t^n) \\
    \mat{I} & -\mat{C}(\vc{\Psi}^n)
  \end{pmatrix}.
\end{equation}
The matrix on the left-hand side of \cref{eq:rosenbrock_wanner_stages} is therefore given by
\begin{equation}\label{eq:mixed_rosenbrock_matrix}
  \mat{M}-\gamma\Delta t^n
  \frac{\partial \vc{\mathcal{F}}}{\partial \vc{y}}(\vc{y}^n,t^n) =
  \begin{pmatrix}
    \mat{I}                  & -\gamma\Delta t^n
    (\partial \vc{\mathcal{R}}/\partial \vc{\Psi})(\vc{\Psi}^n,t^n) \\
    -\gamma\Delta t^n\mat{I} & \gamma\Delta t^n\mat{C}(\vc{\Psi}^n)
  \end{pmatrix},
\end{equation}
where all blocks are diagonal except for the upper-right block, which remains sparse due to the local nature of the LDG discretization, with its particular sparsity pattern determined by the elementwise couplings arising from the LDG numerical trace and numerical flux.
The sparse Jacobian is computed using forward-mode automatic differentiation with graph colouring~\cite{revels_forward_mode_automatic_differentiation_2016,montoison_revisiting_sparse_matrix_coloring_2025}, where the sparsity pattern and colouring are specified according to the LDG discretization.
The resulting matrix in \cref{eq:mixed_rosenbrock_matrix} is factorized once per trial step using KLU, a sparse direct LU solver for non-symmetric linear systems~\cite{davis_klu_2010}.
The resulting factorization is then stored and used to compute $\vc{k}_i^n$ for each right-hand side in \cref{eq:rosenbrock_wanner_stages}.
The same sparse factorization approach is used for the pressure-head formulation, in which $\partial \vc{\mathcal{F}}/\partial \vc{y}$ retains the same sparsity pattern as $\partial \vc{\mathcal{R}}/\partial \vc{\Psi}$, because derivatives of the capacity function contribute only to the diagonal entries of the Jacobian.

\begin{remark}\label{rmk:solvability}
  Taking the Schur complement of the upper-left identity block in \cref{eq:mixed_rosenbrock_matrix}, we obtain the matrix $\gamma\Delta t^n\left( \mat{C}(\vc{\Psi}^n) -\gamma\Delta t^n (\partial \vc{\mathcal{R}}/\partial \vc{\Psi})(\vc{\Psi}^n,t^n) \right)$, which is nonsingular when the condition
  \begin{equation}\label{eq:solvability}
    \lvert \gamma \rvert \Delta t^n
    \left\lVert
    \mat{C}(\vc{\Psi}^n)^{-1}
    \frac{\partial \vc{\mathcal{R}}}{\partial \vc{\Psi}}(\vc{\Psi}^n,t^n)
    \right\rVert < 1
  \end{equation}
  holds under any consistent matrix norm (see, for example, Stewart and Sun~\cite[Chapter~III, Theorem~2.5]{stewart_sun_matrix_perturbation_theory_1990}).
  Because \cref{eq:solvability} can be ensured for sufficiently small $\Delta t^n$ when the capacity function is strictly positive, the linear system in \cref{eq:rosenbrock_wanner_stages} is therefore solvable under the same conditions.
  A similar argument can be used to show that the linear system for the pressure-head formulation is solvable for a strictly positive, continuously differentiable capacity function and a sufficiently small time step, although the time step restriction may differ from that of the mixed form.
\end{remark}

\subsection{Adaptive time integration}\label{sec:adaptive_time_integration}

After computing the trial-step update in \cref{eq:rosenbrock_wanner_update_general}, we estimate the temporal error in water content.
For the mixed formulation, let $\vc{\Theta}^{n}$ and $\vc{\Theta}^{n+1}$ denote the water-content blocks of the primary approximation at the corresponding time levels, and let $\hat{\vc{\Theta}}^{n+1}$ denote the water-content block of the embedded approximation, all defined on the mesh at time $t^n$, before any possible mesh adaptation at time $t^{n+1}$.
For the pressure-head formulation, we analogously define $\vc{\Theta}^{n} \coloneqq \vc{\vartheta}(\vc{\Psi}^{n})$, $\vc{\Theta}^{n+1} \coloneqq \vc{\vartheta}(\vc{\Psi}^{n+1})$, and $\hat{\vc{\Theta}}^{n+1} \coloneqq \vc{\vartheta}(\hat{\vc{\Psi}}^{n+1})$, where the constitutive relation is applied componentwise.
As in Ranocha \etal~\cite[Section~2.3]{ranocha_error_based_step_size_2025}, we compute the normalized error estimate for the trial step as
\begin{equation}\label{eq:temporal_error_estimate}
  E^{n+1} \coloneqq \sqrt{\frac{1}{K^n(N+1)} \sum_{j=1}^{K^n(N+1)} \left|
    \frac{\Theta_j^{n+1}-\hat{\Theta}_j^{n+1}} {\texttt{atol} +
      \texttt{rtol}\max\left(\left|\Theta_j^n\right|,\left|\Theta_j^{n+1}\right|\right)}
    \right|^2 },
\end{equation}
where \texttt{atol} and \texttt{rtol} are user-specified absolute and relative tolerances, respectively, and $K^n$ denotes the number of elements in the mesh at time $t^n$.
We accept the trial step and advance the solution to time $t^{n+1}$ only if $E^{n+1}\leq 1$ and every nodal pressure-head value in the primary approximation is negative.
After an accepted step, a proportional-integral controller (see, for example, Gustafsson \etal~\cite{gustafsson_pi_stepsize_1988}) determines a ratio for the step-size update as
\begin{equation}
  r^{n+1} \coloneqq
  \begin{cases}
    r_{\max}^{n+1}, & E^{n+1}=0, \\
    \min\big(
        r_{\max}^{n+1},
        \max\big(r_{\min}^{n+1}, \sigma (E^{n+1})^{-\beta_1}
    \max(E^n,E^0)^{\beta_2}\big)
    \big),          & \text{otherwise},
  \end{cases}
\end{equation}
where $r_{\min}^{n+1} \in (0,1)$ and $r_{\max}^{n+1}\geq 1$ are the minimum and maximum permitted ratios between time steps, $\beta_1, \beta_2 > 0$ are the current-error and previous-error exponents, $\sigma \in (0, 1]$ is a safety factor, and $E^0 > 0$ is a prescribed initial value for the error estimate.
The time step size is then updated as
\begin{equation}
  \Delta t^{n+1}
  \coloneqq
  \begin{cases}
    \Delta t^n,        & r_{\mathrm{db}}^- \leq r^{n+1} \leq r_{\mathrm{db}}^+, \\
    r^{n+1}\Delta t^n, & \text{otherwise},
  \end{cases}
\end{equation}
where $r_{\mathrm{db}}^-$ and $r_{\mathrm{db}}^+$ are the lower and upper thresholds, respectively, of a controller deadband that suppresses small changes in the step size, with $r_{\mathrm{db}}^- \leq 1 \leq r_{\mathrm{db}}^+$.
If the final time would be overshot by the next time step, the step size is accordingly reduced to reach the final time exactly.
If any nodal pressure-head value in the primary approximation is non-negative (\ie a saturated state is reached), we reject the trial step and repeat the update with a smaller step size obtained by multiplying $\Delta t^n$ by $r_{\min}^{n+1}$.
If the pressure-head values are all negative but the normalized error estimate satisfies $E^{n+1}>1$, we similarly reject the trial step but instead multiply $\Delta t^n$ by $\max(r_{\min}^{n+1},\sigma (E^{n+1})^{-\beta_1})$.
An adaptive mesh update is performed after every $n_{\text{adapt}}$ accepted time steps, with the solution transferred to the new mesh following the procedure described in \cref{sec:conservative_solution_transfer}, preserving the current time step size and error estimate across the mesh update for use in the step-size controller.
Although the controller accepts user-specified absolute minimum and maximum step sizes, we do not impose either bound in this work.

\subsection{Fully discrete mass conservation}\label{sec:fully_discrete_mass_conservation}

The proof of fully discrete mass conservation relies on the preservation of linear invariants by the Rosenbrock--Wanner method, which is well known in the context of ODEs (see, for example, Hundsdorfer and Verwer~\cite[Section~II.1.5]{hundsdorfer_numerical_2003}).
The following lemma demonstrates that this property straightforwardly extends to the DAE system in \cref{eq:temporal_mass_matrix_system} for invariants of the form $\vc{c}^\T\mat{M}\vc{y}$.

\begin{lemma}\label{lem:rosenbrock_linear_invariant}
  Consider a system of the form \cref{eq:temporal_mass_matrix_system} with a constant, possibly singular mass matrix $\mat{M}$ and a right-hand side $\vc{\mathcal{F}}(\vc{y},t)$ that is continuously differentiable with respect to $\vc{y}$ and $t$.
  Suppose that there exists a constant vector $\vc{c}$ such that
  \begin{equation}\label{eq:linear_invariant_condition}
    \vc{c}^\T\vc{\mathcal{F}}(\vc{y},t)=0
  \end{equation}
  for all states $\vc{y}$ and times $t$ within the domain of $\vc{\mathcal{F}}$.
  The Rosenbrock--Wanner method defined by the update formulas in \cref{eq:rosenbrock_wanner_update_general} and the stage equations in \cref{eq:rosenbrock_wanner_stages} then preserves the linear invariant $\vc{c}^\T\mat{M}\vc{y}$, provided that the stage equations are solvable.
\end{lemma}

\begin{proof}
  Following \cite{hundsdorfer_numerical_2003}, we differentiate \cref{eq:linear_invariant_condition} with respect to $\vc{y}$ and $t$ to obtain
  \begin{equation}\label{eq:linear_invariant_derivative_conditions}
    \vc{c}^\T
    \frac{\partial \vc{\mathcal{F}}}{\partial \vc{y}}(\vc{y},t)=\vc{0}^\T,
    \quad
    \vc{c}^\T
    \frac{\partial \vc{\mathcal{F}}}{\partial t}(\vc{y},t)=0.
  \end{equation}
  Left-multiplying the stage equations in \cref{eq:rosenbrock_wanner_stages} by $\vc{c}^\T$ and using \cref{eq:linear_invariant_condition,eq:linear_invariant_derivative_conditions}, we see that each stage solution satisfies
  \begin{equation}\label{eq:linear_invariant_stage}
    \vc{c}^\T\mat{M}\vc{k}_i^n=0.
  \end{equation}
  Left-multiplying the update formula for $\vc{y}^{n+1}$ in \cref{eq:rosenbrock_wanner_update_general} by $\vc{c}^\T\mat{M}$ and substituting \cref{eq:linear_invariant_stage}, we therefore obtain $\vc{c}^\T\mat{M}\vc{y}^{n+1}=\vc{c}^\T\mat{M}\vc{y}^n$.
\end{proof}

The following theorem establishes the fully discrete counterpart of \cref{eq:continuous_water_mass_balance} for the mixed-form discontinuous spectral-element discretization with Rosenbrock--Wanner time integration, including nonzero boundary fluxes.

\begin{theorem}\label{thm:fully_discrete_mass_conservation}
  Consider the mixed formulation under the assumptions of \cref{thm:semidiscrete_mass_conservation}, advanced in time using the Rosenbrock--Wanner method defined by \cref{eq:rosenbrock_wanner_update_general,eq:rosenbrock_wanner_stages}, with the adaptivity performed using the water-content solution transfer described in \cref{sec:conservative_solution_transfer}, and assume that the stage equations are solvable at each time step (see \cref{rmk:solvability}).
  To track the cumulative boundary fluxes, we follow \cite{ireson_openre_2023} and introduce the passive SFOM variables $F_{\mathrm T}(t)$ and $F_{\mathrm B}(t)$, which, at the semi-discrete level, satisfy
  \begin{equation}\label{eq:cumulative_boundary_fluxes}
    \dot{F}_{\mathrm T}(t) = f_{\mathrm T}^\star(\vc{\Psi}(t), t), \quad F_{\mathrm T}(0)=0, \quad \dot{F}_{\mathrm B}(t)=f_{\mathrm B}^\star(\vc{\Psi}(t), t), \quad F_{\mathrm B}(0)=0,
  \end{equation}
  where $f_{\mathrm T}^\star(\vc{\Psi}(t),t)$ and $f_{\mathrm B}^\star(\vc{\Psi}(t),t)$ are the numerical flux values at the top and bottom boundaries expressed as functions of the state $\vc{\Psi}(t)$ and time $t$.
  We integrate \cref{eq:cumulative_boundary_fluxes} in time using the same Rosenbrock--Wanner stages, time steps, and step-acceptance decisions used to advance $\vc{y}(t)$ to obtain the discrete cumulative flux values $F_{\mathrm T}^n$ and $F_{\mathrm B}^n$ at time $t^n$, which are not affected by adaptive mesh updates.
  In exact arithmetic, every accepted time step then satisfies
  \begin{equation}\label{eq:fully_discrete_water_mass_balance}
    \sum_{k=1}^{K^{n+1}}J_k^{n+1}\vc{1}^\T\mat{W}\vc{\theta}_k^{n+1}
    - \sum_{k=1}^{K^n}J_k^n\vc{1}^\T\mat{W}\vc{\theta}_k^n
    = \left(F_{\mathrm B}^{n+1}-F_{\mathrm B}^n\right) - \left(F_{\mathrm T}^{n+1}-F_{\mathrm T}^n\right),
  \end{equation}
  where we use $\vc{\theta}_k^n$ to denote the contribution to the evolved water-content block $\vc{\Theta}^n$ from element $k$ on the mesh obtained after any adaptive mesh update at time $t^n$.
\end{theorem}

\begin{proof}
  Appending the cumulative boundary fluxes in \cref{eq:cumulative_boundary_fluxes} to the mixed-form DAE in \cref{eq:temporal_mixed_time_data} gives an augmented system that takes the form of \cref{eq:temporal_mass_matrix_system} with
  \begin{equation}\label{eq:augmented_mixed_time_data}
    \vc{y} =
    \begin{pmatrix}
      \vc{\Theta} \\
      \vc{\Psi} \\
      F_{\mathrm T} \\
      F_{\mathrm B}
    \end{pmatrix},
    \quad
    \mat{M} = \operatorname{diag}(\mat{I},\mat{0},1,1),
    \quad
    \vc{\mathcal{F}}(\vc{y},t) =
    \begin{pmatrix}
      \boldsymbol{\mathcal{R}}(\vc{\Psi},t) \\
      \vc{\Theta} - \vc{\vartheta}(\vc{\Psi}) \\
      f_{\mathrm T}^\star(\vc{\Psi},t) \\
      f_{\mathrm B}^\star(\vc{\Psi},t)
    \end{pmatrix}.
  \end{equation}
  Consider first the time-integration update from $t^n$ to $t^{n+1}$ on the current mesh, before any scheduled adaptive mesh update at $t^{n+1}$.
  Suppressing time step index superscripts on mesh-related quantities, we define the constant vector
  \begin{equation}\label{eq:water_mass_invariant_vector}
    \vc{c} =
    \begin{pmatrix}
      J_1 \vc{1}^\T \mat{W} & \cdots & J_K \vc{1}^\T \mat{W} & \vc{0}^\T & 1 & -1
    \end{pmatrix}^\T,
  \end{equation}
  where, in this case, $\vc{0}^\T$ denotes a row vector of length $K(N+1)$, indicating the absence of direct contributions from the pressure-head block $\vc{\Psi}$.
  Contracting with the right-hand side $\vc{\mathcal{F}}(\vc{y},t)$, we obtain
  \begin{equation}\label{eq:water_mass_invariant_condition}
    \vc{c}^\T\vc{\mathcal{F}}(\vc{y},t)
    = \begin{pmatrix}
      J_1\vc{1}^\T\mat{W} & \cdots & J_K\vc{1}^\T\mat{W}
    \end{pmatrix}
    \boldsymbol{\mathcal{R}}(\vc{\Psi},t)
    + f_{\mathrm T}^\star(\vc{\Psi},t) - f_{\mathrm B}^\star(\vc{\Psi},t),
  \end{equation}
  which vanishes for every admissible state $\vc{y}$ and time $t$ due to the elementwise balance and interface cancellation established in the proof of \cref{thm:semidiscrete_mass_conservation}.
  Therefore, \cref{lem:rosenbrock_linear_invariant} guarantees that the quantity
  \begin{equation}\label{eq:augmented_water_mass_invariant}
    \vc{c}^\T\mat{M}\vc{y}(t)
    = \sum_{k=1}^{K}J_k\vc{1}^\T\mat{W}\vc{\theta}_k(t) + F_{\mathrm T}(t) - F_{\mathrm B}(t)
  \end{equation}
  is preserved when advancing the solution from $t^n$ to $t^{n+1}$ on a fixed mesh.
  Expanding and rearranging the resulting invariant identity yields \cref{eq:fully_discrete_water_mass_balance} when the solutions at $t^n$ and $t^{n+1}$ are both represented on the mesh at $t^n$, such that $K^{n+1} = K^n$ and $J_k^{n+1} = J_k^n$.
  If a mesh update is performed after advancing the solution to $t^{n+1}$, the mass-conservative solution transfer in \cref{thm:conservative_solution_transfer} preserves the water-content sum, and the cumulative boundary fluxes $F_{\mathrm T}^{n+1}$ and $F_{\mathrm B}^{n+1}$ remain unchanged.
  Therefore, \cref{eq:fully_discrete_water_mass_balance} also holds after the solution transfer to the updated mesh, with $K^{n+1}$ and $J_k^{n+1}$ denoting the number of elements and the Jacobian for element $k$, respectively, on the updated mesh.
\end{proof}

\section{Numerical experiments}\label{sec:numerical_experiments}

The methods described in this paper are implemented in \HydroTrixi,\footnote{The source code is available at \url{https://github.com/uofs-simlab/HydroTrixi.jl}, and the scripts used to generate the numerical results are included in this paper's reproducibility repository.} a Julia package that extends the \Trixi{} numerical simulation framework \cite{ranocha_trixi_2022} to nonlinear parabolic problems for applications in hydrology.
To examine the accuracy, convergence, and mass conservation properties of the mixed and pressure-head formulations described in this work, \HydroTrixi{} is used to solve a series of benchmark problems for the one-dimensional Richards equation, with all simulations using double-precision arithmetic.
The pressure-head formulation is included for comparison purposes; the proposed fully discrete mass-conservative method uses the mixed formulation with water-content solution transfer.

\subsection{Constitutive models}
We consider both the Haverkamp \cite{haverkamp_comparison_infiltration_1977} and van Genuchten--Mualem \cite{van_genuchten_closed_form_1980,mualem_new_model_1976} constitutive models. For the Haverkamp model, the water content and the hydraulic conductivity are given as functions of the pressure head by
\begin{equation}\label{eq:haverkamp_constitutive_model}
  \watercontent(\pressurehead)
  = \theta_{\mathrm{r}} + \frac{\theta_{\mathrm{s}}-\theta_{\mathrm{r}}}
  {1+(a_{\mathrm{H}}\lvert\pressurehead\rvert)^{\beta_{\mathrm{H}}}},
  \quad
  \conductivity(\pressurehead)
  = \frac{\conductivity_{\mathrm{s}}}
  {1+(b_{\mathrm{H}}\lvert\pressurehead\rvert)^{\gamma_{\mathrm{H}}}},
\end{equation}
where $\theta_{\mathrm{r}}$ and $\theta_{\mathrm{s}}$ are the residual and saturated water-content values, respectively, with $\theta_{\mathrm{r}} < \theta_{\mathrm{s}}$, and $\conductivity_{\mathrm{s}}>0$ is the saturated hydraulic conductivity.
The parameters $a_{\mathrm{H}}>0$ and $b_{\mathrm{H}}>0$ are inverse pressure-head scales, and $\beta_{\mathrm{H}}>0$ and $\gamma_{\mathrm{H}}>0$ are the water-retention and hydraulic-conductivity exponents, respectively.
The van Genuchten--Mualem model is given by
\begin{equation}\label{eq:van_genuchten_mualem_constitutive_model}
  \begin{aligned}
    \watercontent(\pressurehead)
    &= \theta_{\mathrm{r}} + \left(\theta_{\mathrm{s}}-\theta_{\mathrm{r}}\right)
    \big[1+(\alpha_{\mathrm{V}}\lvert\pressurehead\rvert)^{n_{\mathrm{V}}}\big]^{-m_{\mathrm{V}}},
 \\
    \conductivity(\pressurehead)
    &= \conductivity_{\mathrm{s}}S_{\mathrm{e}}(\pressurehead)^{\ell_{\mathrm{V}}}
    \big[1-(1-S_{\mathrm{e}}(\pressurehead)^{1/m_{\mathrm{V}}})^{m_{\mathrm{V}}}\big]^2,
  \end{aligned}
\end{equation}
where $S_{\mathrm{e}}(\psi) \coloneqq (\watercontent(\psi)-\theta_{\mathrm{r}})/(\theta_{\mathrm{s}}-\theta_{\mathrm{r}})$ is the effective saturation, $\alpha_{\mathrm{V}}>0$ is an inverse pressure-head scale, $n_{\mathrm{V}}>1$ is a dimensionless water-retention exponent, $\ell_{\mathrm{V}}$ is the dimensionless pore-connectivity parameter, and $m_{\mathrm{V}} \coloneqq 1-1/n_{\mathrm{V}}$ is a derived shape parameter.
The parameter values used in this paper for the Haverkamp and van Genuchten--Mualem models are listed in \cref{tab:constitutive_parameters}.
The Haverkamp values are chosen to reproduce the first infiltration benchmark considered by Celia \etal~\cite{celia_mass_conservative_richards_1990}, whereas the van Genuchten--Mualem values correspond to the New Mexico soil properties used in their second infiltration benchmark.

\begin{table}[t]
  \centering
  \caption{Parameters for the Haverkamp and van Genuchten--Mualem constitutive models used in the numerical experiments.}
  \label{tab:constitutive_parameters}
  \begin{tabular*}{\textwidth}{@{\extracolsep{\fill}}lll@{}}
    \toprule
    Symbol & Description & Value \\
    \midrule
    \multicolumn{3}{l}{\emph{Haverkamp constitutive model}} \\
    $\theta_{\mathrm{r}}$ & Residual water content & $0.075$ \\
    $\theta_{\mathrm{s}}$ & Saturated water content & $0.287$ \\
    $\conductivity_{\mathrm{s}}$ & Saturated conductivity & $9.44\times 10^{-5}\,\mathrm{m}\,\mathrm{s}^{-1}$ \\
    $a_{\mathrm{H}}$ & Inverse pressure-head scale & $2.70740\,\mathrm{m}^{-1}$ \\
    $b_{\mathrm{H}}$ & Inverse pressure-head scale & $5.24084\,\mathrm{m}^{-1}$ \\
    $\beta_{\mathrm{H}}$ & Water-retention exponent & $3.96$ \\
    $\gamma_{\mathrm{H}}$ & Conductivity exponent & $4.74$ \\
    \midrule
    \multicolumn{3}{l}{\emph{van Genuchten--Mualem constitutive model}} \\
    $\theta_{\mathrm{r}}$ & Residual water content & $0.102$ \\
    $\theta_{\mathrm{s}}$ & Saturated water content & $0.368$ \\
    $\conductivity_{\mathrm{s}}$ & Saturated conductivity & $9.22\times 10^{-5}\,\mathrm{m}\,\mathrm{s}^{-1}$ \\
    $\alpha_{\mathrm{V}}$ & Inverse pressure-head scale & $3.35\,\mathrm{m}^{-1}$ \\
    $\ell_{\mathrm{V}}$ & Pore-connectivity parameter & $0.5$ \\
    $n_{\mathrm{V}}$ & Water-retention exponent & $2$ \\
    \bottomrule
  \end{tabular*}
\end{table}

\subsection{Spatial and temporal convergence for a manufactured solution}

We assess spatial and temporal convergence for fixed meshes and time step sizes using two variants of the manufactured-solution problem introduced by Keita \etal~\cite{keita_time_stepping_richards_2021}.
Both variants use the Haverkamp model with the parameters in \cref{tab:constitutive_parameters} and prescribe the exact pressure-head profile
\begin{equation}\label{eq:manufactured_pressure_head}
  \pressurehead(z,t)
  = 0.204\,\mathrm{m}\,\tanh\left[
    \frac{1}{2}\left(100\,\mathrm{m}^{-1}z+\frac{t}{12\,\mathrm{s}}-15\right)
    \right]-0.411\,\mathrm{m},
\end{equation}
on a column of length $L = 0.2\,\mathrm{m}$ for $t\in(0\,\mathrm{s},120\,\mathrm{s}]$.
To produce the manufactured solution, a source term given by
\begin{equation}\label{eq:manufactured_source_term}
  s(z,t)
  \coloneqq
  \capacity(\pressurehead(z,t))
  \partial_t\pressurehead(z,t)
  -\partial_z\left[
    \conductivity(\pressurehead(z,t))
    \left(\partial_z\pressurehead(z,t)-1\right)
    \right],
\end{equation}
is added to the right-hand side of \cref{eq:richards_equation}, where $\pressurehead(z,t)$ is the manufactured solution profile given by \cref{eq:manufactured_pressure_head}.
In both variants, we obtain the initial condition by evaluating this profile at $t=0 \, \mathrm{s}$, and we specify the corresponding exact pressure head $\psi(0 \,\mathrm{m},t)$ as a Dirichlet condition at the top boundary.
At the bottom boundary, the first variant prescribes the exact pressure head $\pressurehead(L,t)$, whereas the second prescribes the exact flux $f(\pressurehead(L,t),\partial_z\pressurehead(L,t))$.
We refer to these variants as the \emph{Dirichlet--Dirichlet} and \emph{Dirichlet--Neumann} problems, respectively.
\par
To examine spatial convergence, we consider fixed polynomial degrees $N \in \{3,4\}$ and perform successive uniform refinements from $l=4$ to $l=8$, corresponding to $K \in \{16, 32, 64, 128, 256\}$.
We use a fixed time step size of $\Delta t = 0.1 \,\mathrm{s}$ to ensure that temporal errors remain small compared to spatial errors.
The nominal mesh size is defined by $\Delta z \coloneqq L/(KN)$, corresponding to the average distance between the LGL quadrature nodes in physical space, and we define the normalized quadrature-based $L^2$ error and the nodal $L^\infty$ error in pressure head at $t = 120 \, \mathrm{s}$ as
\begin{subequations}\label{eq:pressure_head_error_norms}
  \begin{align}
    \text{$L^2$ pressure-head error}
    &\coloneqq
    \sqrt{
      \frac{1}{L}
      \sum_{k=1}^K J_k \sum_{i=0}^N \omega_i
      \left(
      \pressurehead(X_k(\xi_i),t)
      - \pressurehead_k^N(\xi_i,t)
      \right)^2
    }, \label{eq:pressure_head_L2_error} \\
    \text{$L^\infty$ pressure-head error}
    &\coloneqq
    \max_{k=1}^K
    \,
    \max_{i=0}^N
    \left|
    \pressurehead(X_k(\xi_i),t)
    - \pressurehead_k^N(\xi_i,t)
    \right|, \label{eq:pressure_head_Linf_error}
  \end{align}
\end{subequations}
respectively, where both norms return values in the same units as the pressure head (i.e., metres for the examples considered in this paper) under this normalization.
In \cref{fig:manufactured_solution_spatial_convergence}, we plot both error norms against $\Delta z$ for each combination of boundary condition configuration and polynomial degree, using the mixed and pressure-head formulations.
In all cases considered, the $L^2$ and $L^\infty$ error norms both converge at an order of about $N+1$, which is optimal when discretizing a smooth problem using a polynomial approximation of degree $N$.
Moreover, the mixed and pressure-head formulations produce nearly identical pressure-head errors at every refinement level, which is to be expected because they make use of the same semi-discrete spatial operator but differ in the variables that they evolve in time.

\begin{figure}[tp]
  \centering
  \begin{subfigure}{0.5\textwidth}
    \includegraphics[width=0.96\textwidth]{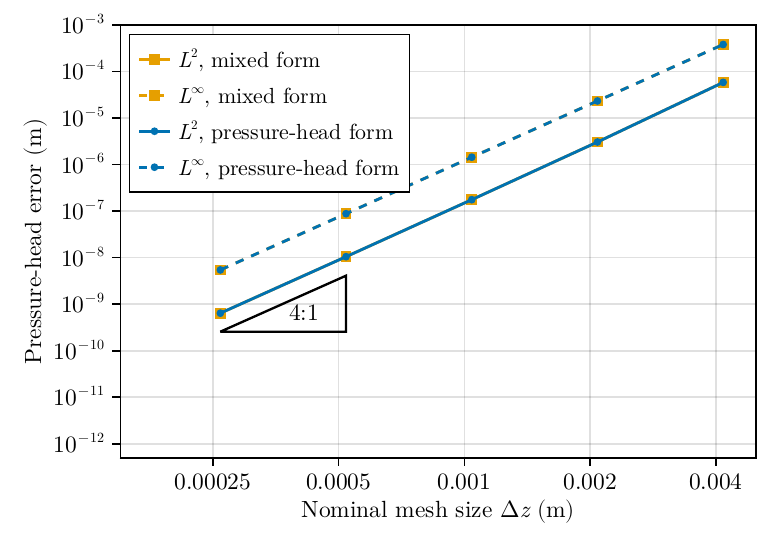}
    \subcaption{Dirichlet--Dirichlet, $N = 3$}\label{fig:manufactured_solution_dirichlet_dirichlet_spatial_convergence_N3}
  \end{subfigure}\hfill
  \begin{subfigure}{0.5\textwidth}
    \includegraphics[width=0.96\textwidth]{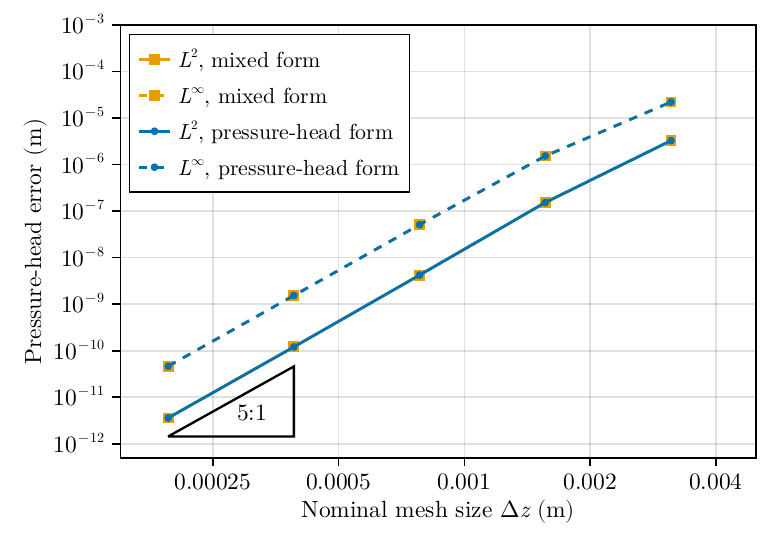}
    \subcaption{Dirichlet--Dirichlet, $N = 4$}\label{fig:manufactured_solution_dirichlet_dirichlet_spatial_convergence_N4}
  \end{subfigure} \\
  \begin{subfigure}{0.5\textwidth}
    \includegraphics[width=0.96\textwidth]{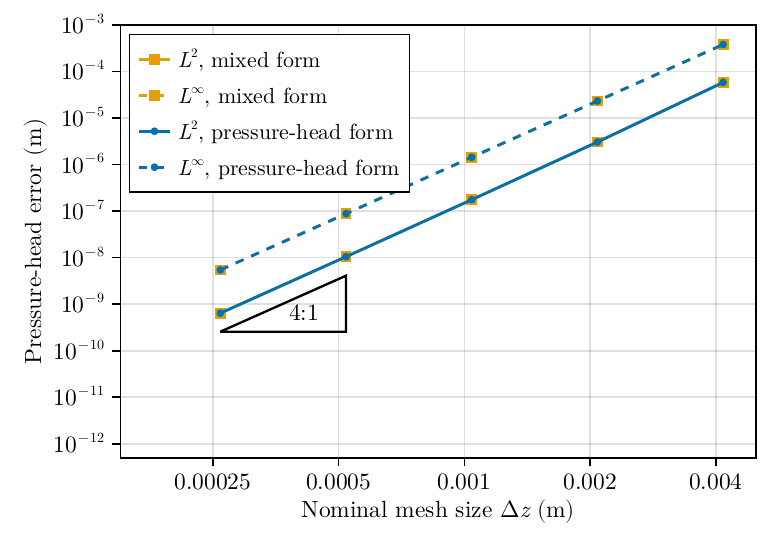}
    \subcaption{Dirichlet--Neumann, $N = 3$}\label{fig:manufactured_solution_dirichlet_neumann_spatial_convergence_N3}
  \end{subfigure}\hfill
  \begin{subfigure}{0.5\textwidth}
    \includegraphics[width=0.96\textwidth]{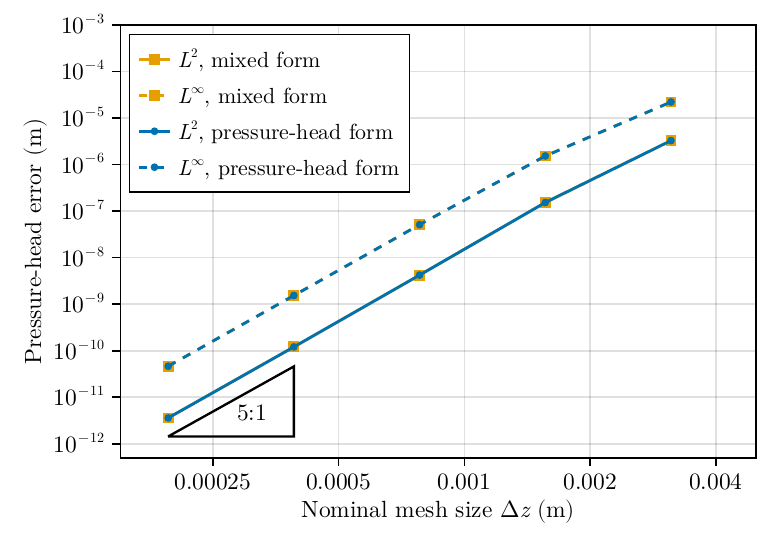}
    \subcaption{Dirichlet--Neumann, $N = 4$}\label{fig:manufactured_solution_dirichlet_neumann_spatial_convergence_N4}
  \end{subfigure}%
  \caption{Spatial convergence of the pressure-head error at $t=120\,\mathrm{s}$ for the manufactured-solution problem using a fixed time step size of $\Delta t=0.1\,\mathrm{s}$.}\label{fig:manufactured_solution_spatial_convergence}
\end{figure}

To assess the effect of the choice between the mixed formulation and the pressure-head formulation on the accuracy of the scheme, we consider a fixed spatial discretization with $N=3$ and $K=4096$, which is fine enough to ensure that the spatial error is negligible compared to the temporal error, and we successively halve the time step to observe the temporal convergence for both formulations.
The convergence of the pressure-head error for both boundary condition configurations is plotted in \cref{fig:manufactured_solution_temporal_convergence} for fixed time step sizes ranging from $\Delta t = 4\,\mathrm{s}$ to $\Delta t = 0.25 \,\mathrm{s}$.
Between the smallest and second-smallest time step sizes, both formulations exhibit temporal rates of convergence between fourth and fifth order, consistent with the reduction in order noted in \cite{steinebach_rodas5p_2023} when applying the nominally fifth-order \texttt{Rodas5P} scheme to stiff problems.
Although the order reduction is stronger for the mixed formulation than for the pressure-head formulation at smaller time step sizes, the mixed formulation is nevertheless more accurate for the entire range of time step sizes considered.

\begin{figure}[tp]
  \centering
  \begin{subfigure}{0.5\textwidth}
    \includegraphics[width=0.96\textwidth]{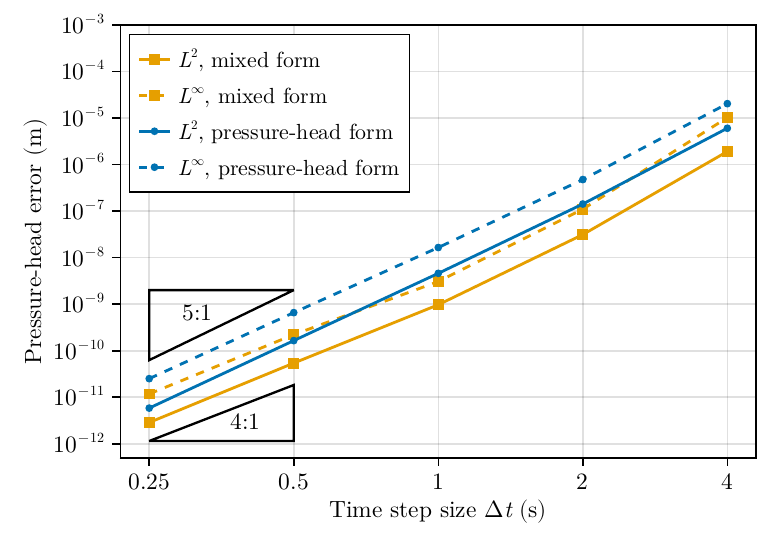}
    \subcaption{Dirichlet--Dirichlet}\label{fig:manufactured_solution_dirichlet_dirichlet_temporal_convergence}
  \end{subfigure}\hfill
  \begin{subfigure}{0.5\textwidth}
    \includegraphics[width=0.96\textwidth]{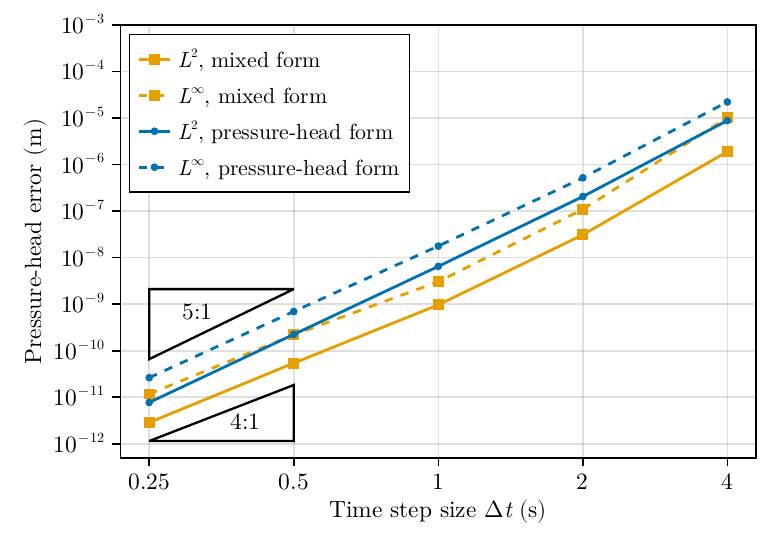}
    \subcaption{Dirichlet--Neumann}\label{fig:manufactured_solution_dirichlet_neumann_temporal_convergence}
  \end{subfigure}%
  \caption{Temporal convergence of the pressure-head error at $t=120\,\mathrm{s}$ for the manufactured-solution problem using a fixed spatial discretization with $N=3$ and $K=4096$.}\label{fig:manufactured_solution_temporal_convergence}
\end{figure}

\subsection{Adaptive simulations of infiltration benchmark problems}

To assess mass conservation under spatial and temporal adaptivity, we consider the two infiltration problems described in \cite{celia_mass_conservative_richards_1990}.
For the first problem, adapted from the sand infiltration case considered in \cite{haverkamp_comparison_infiltration_1977}, we use the Haverkamp parameters in \cref{tab:constitutive_parameters} for a soil column of length $L=0.4\,\mathrm{m}$ and advance the solution until a final time of $t=360\,\mathrm{s}$.
The initial condition is given by $\pressurehead(z,0\,\mathrm{s})=-0.615\,\mathrm{m}$ for $z\in[0\,\mathrm{m},0.4\,\mathrm{m}]$, and the time-independent Dirichlet boundary conditions are given by $\pressurehead(0\,\mathrm{m},t)=-0.207\,\mathrm{m}$ and $\pressurehead(L,t)=-0.615\,\mathrm{m}$ for $t\in(0\,\mathrm{s},360\,\mathrm{s}]$.
We also consider the second infiltration problem from \cite{celia_mass_conservative_richards_1990}, which uses the van Genuchten--Mualem parameters in \cref{tab:constitutive_parameters} for a soil column of length $L=1\,\mathrm{m}$,\footnote{The domain used here matches the depth range shown in \cite[Figure~3]{celia_mass_conservative_richards_1990}, whereas the accompanying text places the lower boundary at a depth of $0.6\,\mathrm{m}$.} advancing until a final time of $t=86400\,\mathrm{s}$.
The initial condition is given by $\pressurehead(z,0\,\mathrm{s})=-10\,\mathrm{m}$ for $z\in[0\,\mathrm{m},1\,\mathrm{m}]$, and the time-independent Dirichlet boundary conditions are given by $\pressurehead(0\,\mathrm{m},t)=-0.75\,\mathrm{m}$ and $\pressurehead(1\,\mathrm{m},t)=-10\,\mathrm{m}$ for $t\in(0\,\mathrm{s},86400\,\mathrm{s}]$.
For ease of reference, we call these the \emph{Haverkamp infiltration problem} and the \emph{New Mexico infiltration problem}, respectively, with the latter name reflecting that its soil parameters approximate the properties of a field site in New Mexico~\cite{celia_mass_conservative_richards_1990}.
The discretization and adaptivity parameters used for both infiltration problems are summarized in \cref{tab:infiltration_adaptivity_parameters}.
We do not claim that these choices are optimal for either problem, but they nevertheless provide realistic settings for assessing mass conservation under spatial and temporal adaptivity, which, as a consequence of \cref{thm:fully_discrete_mass_conservation}, is expected to hold for the mixed formulation with water-content solution transfer independently of the specific adaptivity parameters chosen.
\begin{table}[t]
  \centering
  \caption{Discretization and adaptivity parameters for the Haverkamp and New Mexico infiltration problems.}
  \label{tab:infiltration_adaptivity_parameters}
  \begin{tabular*}{\textwidth}{@{\extracolsep{\fill}}lll@{}}
    \toprule
    Symbol & Description & Value \\
    \midrule
    \multicolumn{3}{l}{\emph{Spatial discretization and adaptivity}} \\
    $N$ & Polynomial degree & $3$ \\
    $l_{\mathrm{init}}$ & Initial refinement level & $6$ \\
    $l_{\min}$ & Minimum refinement level & $2$ \\
    $l_{\max}$ & Maximum refinement level & $10$ \\
    $\eta_{\mathrm{c}}$ & Coarsening threshold & $0.003$ \\
    $\eta_{\mathrm{r}}$ & Refinement threshold & $0.03$ \\
    $\epsilon$ & Indicator regularization parameter & $10^{-11}$ \\
    $n_{\mathrm{adapt}}$ & Accepted time steps between mesh updates & $10$ \\
    \midrule
    \multicolumn{3}{l}{\emph{Temporal discretization and adaptivity}} \\
    $\beta_1$ & Current-error exponent & $0.14$ \\
    $\beta_2$ & Previous-error exponent & $0.08$ \\
    $\sigma$ & Safety factor & $0.9$ \\
    $E^0$ & Initial error value & $10^{-4}$ \\
    $r_{\min}^n$ & Minimum step ratio ($n\geq 1$) & $0.2$ \\
    $r_{\max}^1$ & Initial maximum step ratio & $10^4$ \\
    $r_{\max}^n$ & Maximum step ratio ($n\geq 2$) & $10$ \\
    $r_{\mathrm{db}}^-$ & Lower deadband threshold & $5/6$ \\
    $r_{\mathrm{db}}^+$ & Upper deadband threshold & $1$ \\
    $\Delta t^0$ & Initial time step & $10^{-2}\,\mathrm{s}$ \\
    $\texttt{atol}$ & Absolute temporal error tolerance & $10^{-11}$ \\
    $\texttt{rtol}$ & Relative temporal error tolerance & $10^{-5}$, $10^{-7}$, $10^{-9}$ \\
    \bottomrule
  \end{tabular*}
\end{table}

\Cref{fig:infiltration_profiles} shows the solution profiles for the Haverkamp and New Mexico infiltration problems at intermediate and final times, using the mixed formulation with water-content solution transfer and a relative temporal error tolerance of $10^{-7}$.
The grey vertical lines indicate element boundaries, and we see that using the total-variation-based refinement indicator results in a mesh that adapts to the evolving solution and effectively tracks the moving front as infiltration progresses.
The time step histories for both problems are shown in \cref{fig:infiltration_time_steps} for relative tolerances of $10^{-5}$, $10^{-7}$, and $10^{-9}$.
As expected, the accepted time step sizes are smaller for tighter relative tolerances.
Sharp decreases in accepted time step sizes are observed immediately following mesh adaptation events.
These decreases are consistent with transient temporal error introduced by transferring the solution between meshes.
Similar interactions between temporal error control and mesh adaptivity have been noted in other investigations, with Dolejší and May~\cite{dolejsi_anisotropic_hp_mesh_2023} introducing an automatic reduction of the time step following mesh adaptation to avoid rejected steps.

\begin{figure}[t]
  \centering
  \begin{subfigure}{0.5\textwidth}
    \includegraphics[width=0.96\textwidth]{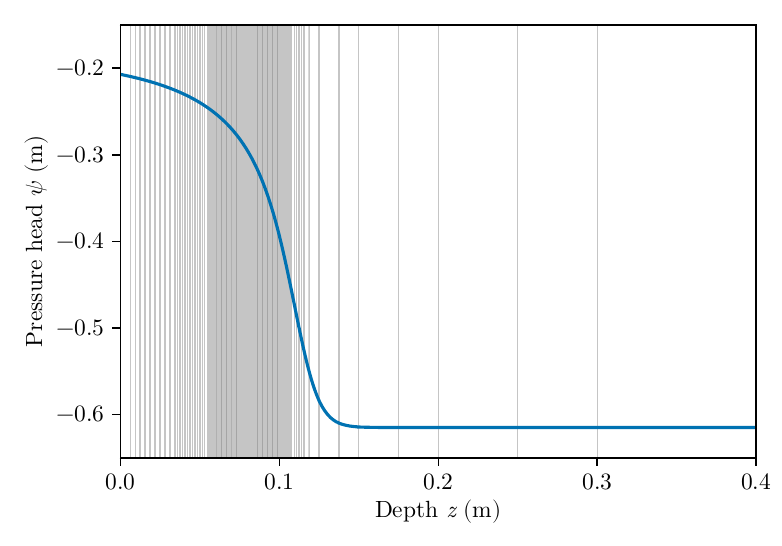}
    \subcaption{Haverkamp, $t=180\,\mathrm{s}$}\label{fig:haverkamp_infiltration_profile_half}
  \end{subfigure}\hfill
  \begin{subfigure}{0.5\textwidth}
    \includegraphics[width=0.96\textwidth]{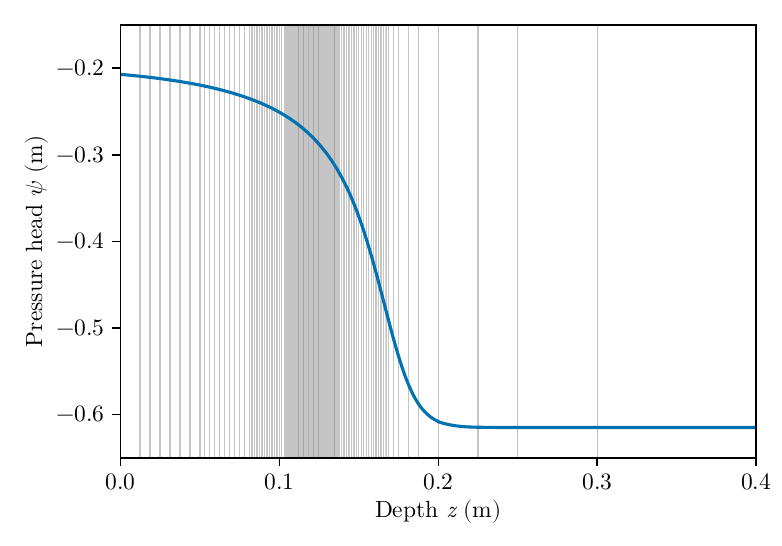}
    \subcaption{Haverkamp, $t=360\,\mathrm{s}$}\label{fig:haverkamp_infiltration_profile_final}
  \end{subfigure} \\
  \begin{subfigure}{0.5\textwidth}
    \includegraphics[width=0.96\textwidth]{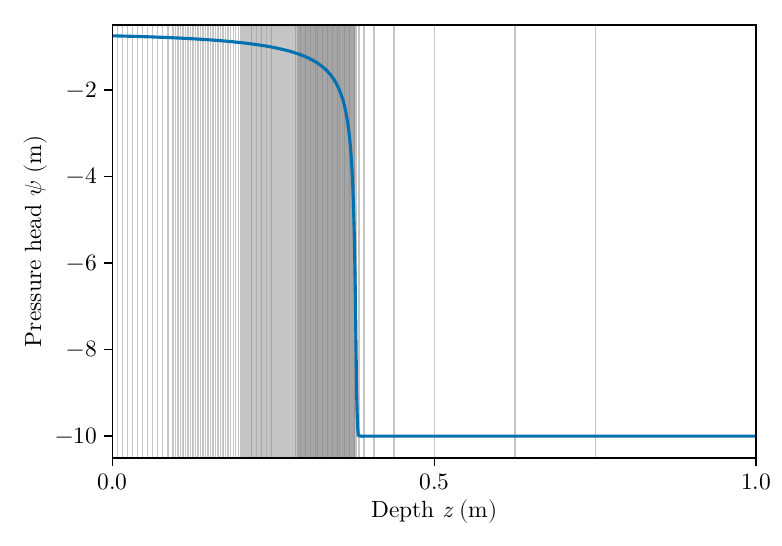}
    \subcaption{New Mexico, $t=43200\,\mathrm{s}$}\label{fig:new_mexico_infiltration_profile_half}
  \end{subfigure}\hfill
  \begin{subfigure}{0.5\textwidth}
    \includegraphics[width=0.96\textwidth]{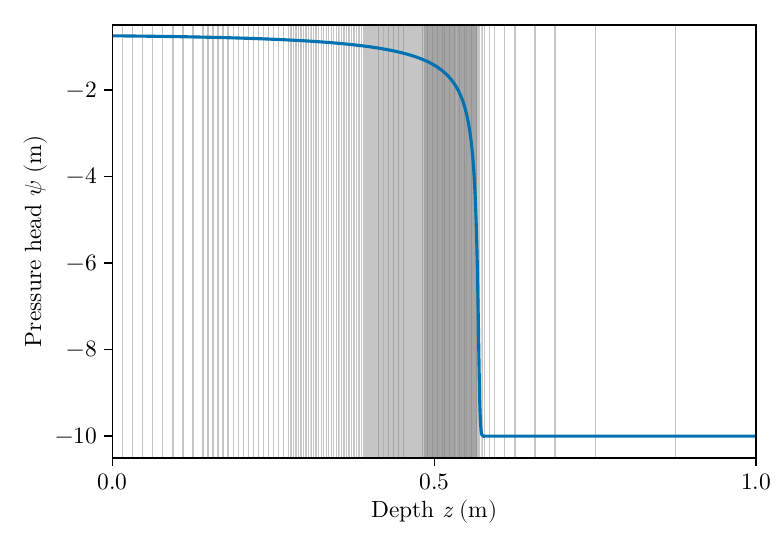}
    \subcaption{New Mexico, $t=86400\,\mathrm{s}$}\label{fig:new_mexico_infiltration_profile_final}
  \end{subfigure}%
  \caption{Pressure-head profiles and element boundaries for the Haverkamp and New Mexico infiltration problems using the mixed formulation with water-content solution transfer and a relative temporal error tolerance of $10^{-7}$.
  }\label{fig:infiltration_profiles}
\end{figure}

\begin{figure}[tp]
  \centering
  \begin{subfigure}{0.5\textwidth}
    \includegraphics[width=0.96\textwidth]{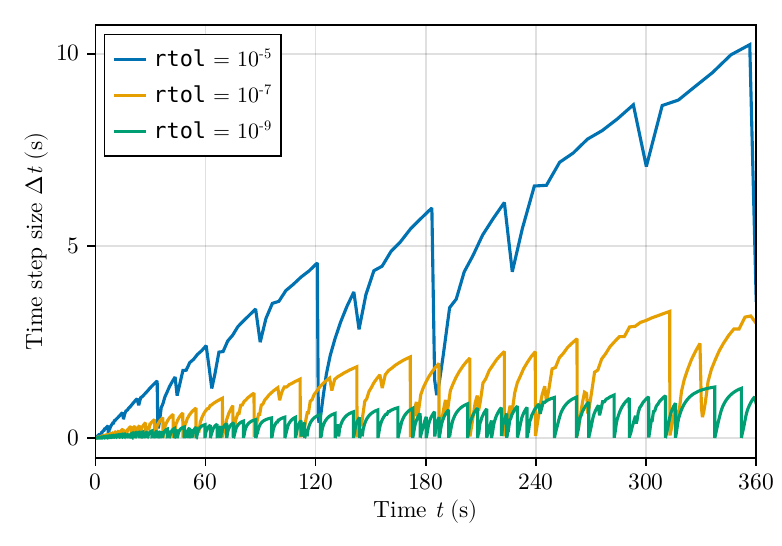}
    \subcaption{Haverkamp}\label{fig:haverkamp_infiltration_time_steps}
  \end{subfigure}\hfill
  \begin{subfigure}{0.5\textwidth}
    \includegraphics[width=0.96\textwidth]{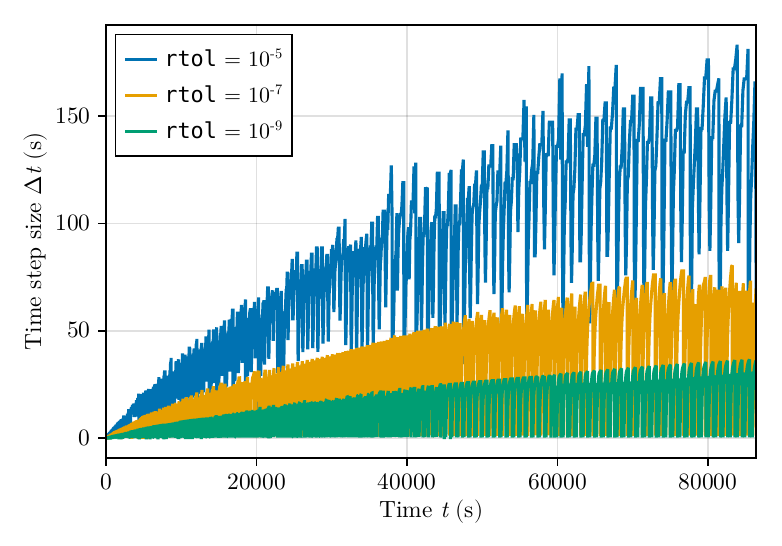}
    \subcaption{New Mexico}\label{fig:new_mexico_infiltration_time_steps}
  \end{subfigure}%
  \caption{Accepted time step sizes for the Haverkamp and New Mexico infiltration problems using the mixed formulation with water-content solution transfer and three different relative temporal error tolerances.}\label{fig:infiltration_time_steps}
\end{figure}

\begin{figure}[tp]
  \centering
  \begin{subfigure}{0.5\textwidth}
    \includegraphics[width=0.96\textwidth]{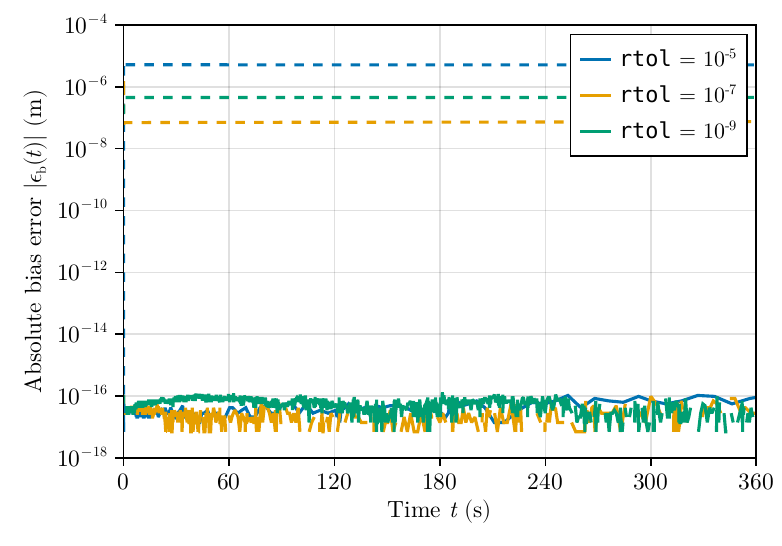}
    \subcaption{Haverkamp, mixed formulation}\label{fig:haverkamp_mixed_mass_bias}
  \end{subfigure}\hfill
  \begin{subfigure}{0.5\textwidth}
    \includegraphics[width=0.96\textwidth]{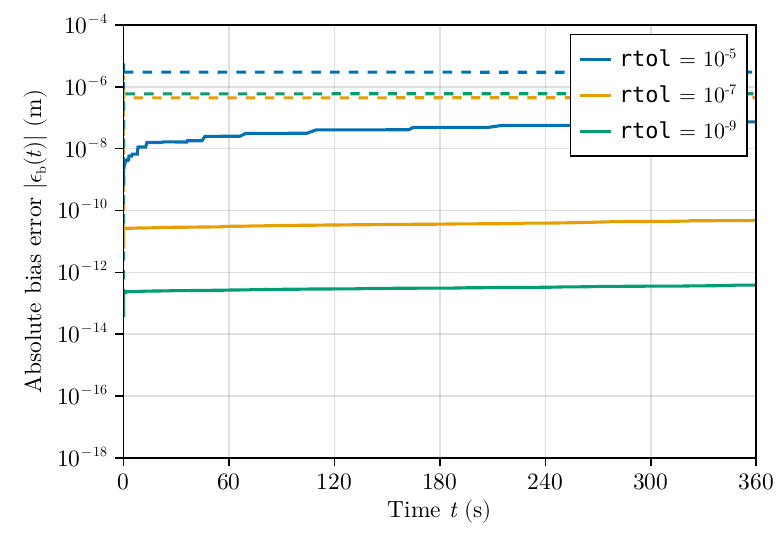}
    \subcaption{Haverkamp, pressure-head formulation}\label{fig:haverkamp_pressure_head_mass_bias}
  \end{subfigure} \\
  \begin{subfigure}{0.5\textwidth}
    \includegraphics[width=0.96\textwidth]{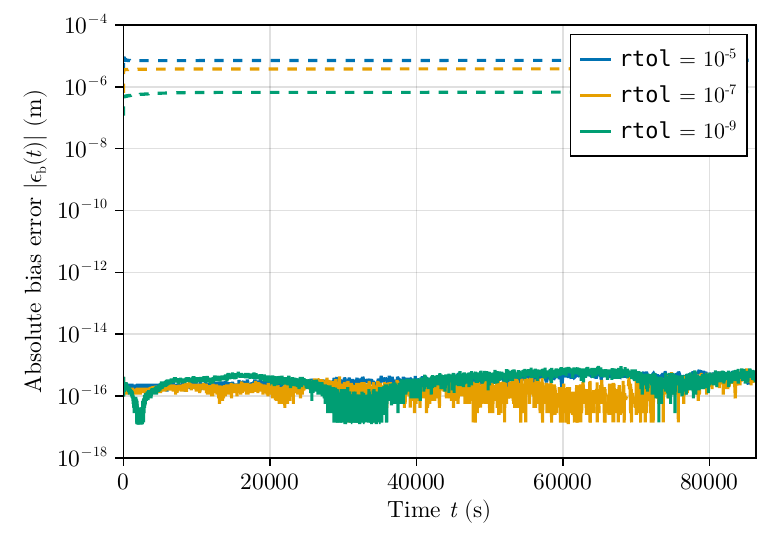}
    \subcaption{New Mexico, mixed formulation}\label{fig:new_mexico_mixed_mass_bias}
  \end{subfigure}\hfill
  \begin{subfigure}{0.5\textwidth}
    \includegraphics[width=0.96\textwidth]{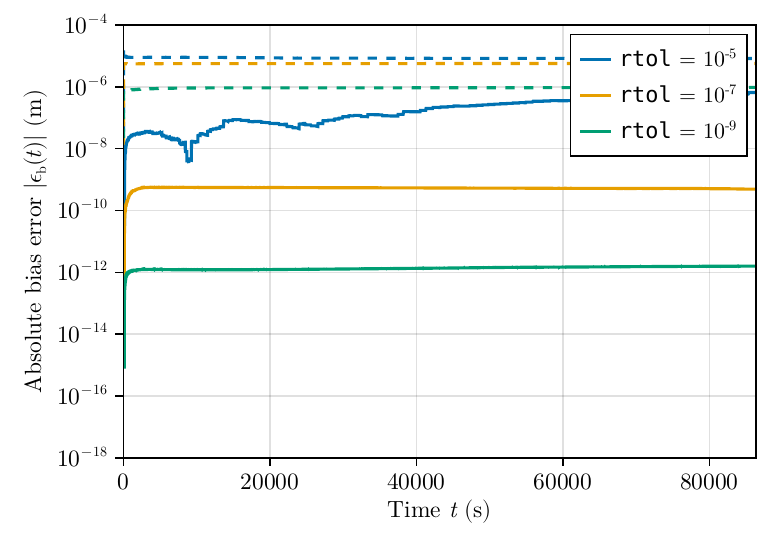}
    \subcaption{New Mexico, pressure-head formulation}\label{fig:new_mexico_pressure_head_mass_bias}
  \end{subfigure}%
  \caption{Evolution of the absolute bias error for the Haverkamp and New Mexico infiltration problems with three different relative temporal error tolerances.
    Solid lines denote water-content solution transfer, with gaps indicating bias errors that are zero in double precision; dashed lines denote pressure-head solution transfer.
  }\label{fig:infiltration_mass_bias}
\end{figure}

Following \cite{ireson_openre_2023}, we quantify mass conservation using the \emph{bias error} and plot its absolute value in \cref{fig:infiltration_mass_bias} for each combination of formulation (i.e., mixed form or pressure-head form) and solution transfer variable (i.e., water content or pressure head).
Using the SFOM variables and notation of \cref{sec:fully_discrete_mass_conservation}, the fully discrete bias error at time $t^n$ is given by
\begin{equation}\label{eq:fully_discrete_mass_bias}
  \begin{aligned}
    \epsilon_{\mathrm{b}}(t^n)
    &\coloneqq
    \left(F_{\mathrm B}^n-F_{\mathrm B}^0\right) - \left(F_{\mathrm T}^n-F_{\mathrm T}^0\right) - \left(
    \sum_{k=1}^{K^n}J_k^n\vc{1}^\T\mat{W}\vc{\theta}_k^n
    - \sum_{k=1}^{K^0}J_k^0\vc{1}^\T\mat{W}\vc{\theta}_k^0
    \right),
  \end{aligned}
\end{equation}
where the water-content vectors $\vc{\theta}_k^0$ and $\vc{\theta}_k^n$ are evaluated on the respective meshes at initialization and at time $t^n$, respectively.
For the mixed formulation, these vectors are obtained from the evolved water-content block of the DAE solution, whereas for the pressure-head formulation, they are computed from the pressure-head ODE solution using the constitutive relation.
As expected from \cref{thm:fully_discrete_mass_conservation}, the mixed formulation preserves the mass balance to roundoff in both problems for all three relative tolerances, provided that the solution transfer is performed using the water content following the conservative transfer procedure described in \cref{thm:conservative_solution_transfer}; transferring the pressure head instead of the water content results in a loss of mass conservation.
For the pressure-head formulation, transferring the water content reduces the mass bias relative to transferring the pressure head, but the bias error does not reach roundoff and remains dependent on the temporal error tolerance.

\section{Conclusions}\label{sec:conclusions}

In this work, we presented a high-order discretization of the one-dimensional Richards equation that supports adaptivity in space and time and ensures mass conservation throughout the simulation.
The method combines a collocated LDG spatial discretization of the mixed-form Richards equation with a high-order Rosenbrock--Wanner time integration scheme and a mass-conservative solution transfer approach.
Numerical experiments based on an open-source Julia implementation confirm high-order spatial and temporal convergence, assess the influence of the temporal formulation on accuracy, and demonstrate the importance of evolving and transferring water content for mass conservation.
\par
This paper provides a foundation for further development of adaptive high-order numerical methods for land models that conserve key physical invariants.
Future work will extend the methodology to more complex hydrologic scenarios, including those involving saturated and near-saturated states, heterogeneous soils, solution-dependent boundary conditions, lateral flow, and coupled mass--energy transport.
Representative test problems, including those proposed by Clark \etal~\cite{clark_laugh_tests_2021}, will provide benchmarks for assessing these extensions, with the central methodological challenge being to retain high-order accuracy and fully discrete mass conservation under spatial and temporal adaptivity while preserving the additional invariants and balance laws of the extended models.

\section*{Acknowledgments}

Funding for this work was provided by the National Oceanic and Atmospheric Administration (NOAA), awarded to the Cooperative Institute for Research to Operations in Hydrology (CIROH) through the NOAA Cooperative Agreement with The University of Alabama, NA22NWS4320003. 
This work was also supported by the Natural Sciences and Engineering Research Council of Canada (NSERC) through a Canada Postdoctoral Research Award.
The statements, findings, conclusions, and recommendations are those of the authors and do not necessarily reflect the views of NOAA or NSERC.
OpenAI Codex and GitHub Copilot were used to assist with the development of the Julia code used for the numerical experiments and with locating relevant literature, typesetting, and manuscript editing.
The authors checked every incorporated suggestion for correctness and take full responsibility for the code and manuscript content.

\section*{Reproducibility of computational results}

The scripts used to generate the results and figures in this paper are provided in the reproducibility repository \cite{montoya_spiteri_adaptive_richards_reproducibility_2026}.

\bibliographystyle{siamplain}
\bibliography{references}

@book{richardson_weather_prediction_1922,
    author = {Richardson, Lewis F.},
    title = {Weather prediction by numerical process},
    publisher = {Cambridge University Press},
    year = {1922}
  }

@article{richards_capillary_1931,
    author = {Richards, Lorenzo A.},
    title = {Capillary conduction of liquids through porous mediums},
    journal = {Physics},
    year = {1931},
    volume = {1},
    number = {5},
    pages = {318--333},
    doi = {10.1063/1.1745010}
  }

@article{celia_mass_conservative_richards_1990,
    author = {Celia, Michael A. and Bouloutas, Efthimios T. and
        Zarba, Rebecca L.},
    title = {A general mass-conservative numerical solution for the unsaturated
        flow equation},
    journal = {Water Resources Research},
    year = {1990},
    volume = {26},
    number = {7},
    pages = {1483--1496},
    doi = {10.1029/WR026i007p01483}
  }

@article{tocci_accurate_economical_1997,
    author = {Tocci, Michael D. and Kelley, C. T. and Miller, Cass T.},
    title = {Accurate and economical solution of the pressure-head form of
          {Richards' equation} by the method of lines},
    journal = {Advances in Water Resources},
    year = {1997},
    volume = {20},
    number = {1},
    pages = {1--14},
    doi = {10.1016/S0309-1708(96)00008-5}
  }

@article{cockburn_ldg_1998,
    author = {Cockburn, Bernardo and Shu, Chi-Wang},
    title = {The local discontinuous {Galerkin} method for time-dependent
        convection-diffusion systems},
    journal = {SIAM Journal on Numerical Analysis},
    year = {1998},
    volume = {35},
    number = {6},
    pages = {2440--2463},
    doi = {10.1137/S0036142997316712}
  }

@article{cockburn_minimal_dissipation_2007,
    author = {Cockburn, Bernardo and Dong, Bo},
    title = {An analysis of the minimal dissipation local discontinuous
          {Galerkin} method for convection-diffusion problems},
    journal = {Journal of Scientific Computing},
    year = {2007},
    volume = {32},
    number = {2},
    pages = {233--262},
    doi = {10.1007/s10915-007-9130-3}
  }

@article{dolejsi_anisotropic_hp_mesh_2023,
    author = {Dolej{\v{s}}{\'{i}}, V{\'{i}}t and May, Georg},
    title = {An anisotropic {$hp$}-mesh adaptation method for time-dependent
        problems based on interpolation error control},
    journal = {Journal of Scientific Computing},
    year = {2023},
    volume = {95},
    number = {2},
    eid = {36},
    doi = {10.1007/s10915-023-02153-1}
  }

@article{miller_spatially_temporally_adaptive_2006,
    author = {Miller, Cass T. and Abhishek, Chandra and Farthing, Matthew W.},
    title = {A spatially and temporally adaptive solution of
          {Richards' equation}},
    journal = {Advances in Water Resources},
    year = {2006},
    volume = {29},
    number = {4},
    pages = {525--545},
    doi = {10.1016/j.advwatres.2005.06.008}
  }

@article{li_adaptive_ldg_2007,
    author = {Li, H. and Farthing, M. W. and Miller, C. T.},
    title = {Adaptive local discontinuous {Galerkin} approximation to
          {Richards' equation}},
    journal = {Advances in Water Resources},
    year = {2007},
    volume = {30},
    number = {9},
    pages = {1883--1901},
    doi = {10.1016/j.advwatres.2007.02.007}
  }

@article{li_local_ldg_2007,
    author = {Li, H. and Farthing, M. W. and Dawson, C. N. and Miller, C. T.},
    title = {Local discontinuous {Galerkin} approximations to
          {Richards' equation}},
    journal = {Advances in Water Resources},
    year = {2007},
    volume = {30},
    number = {3},
    pages = {555--575},
    doi = {10.1016/j.advwatres.2006.04.011}
  }

@article{dolejsi_adaptive_space_time_2019,
    author = {Dolej{\v{s}}{\'{i}}, V{\'{i}}t and Kuraz, Michal and Solin, Pavel},
    title = {Adaptive higher-order space-time discontinuous {Galerkin} method
        for the computer simulation of variably-saturated porous media flows},
    journal = {Applied Mathematical Modelling},
    year = {2019},
    volume = {72},
    pages = {276--305},
    doi = {10.1016/j.apm.2019.02.037}
  }

@article{dolejsi_error_estimates_adaptivity_2024,
    author = {Dolej{\v{s}}{\'{i}}, V{\'{i}}t and Shin, Hyun-Geun and
        Vlas{\'{a}}k, Miloslav},
    title = {Error estimates and adaptivity of the space-time discontinuous
          {Galerkin} method for solving the {Richards} equation},
    journal = {Journal of Scientific Computing},
    year = {2024},
    volume = {101},
    number = {1},
    eid = {11},
    doi = {10.1007/s10915-024-02650-x}
  }

@article{clement_adaptive_dg_2021,
    author = {Cl{\'e}ment, Jean-Baptiste and Golay, Fr{\'e}d{\'e}ric and
        Ersoy, Mehmet and Sous, Damien},
    title = {An adaptive strategy for discontinuous {Galerkin} simulations of
          {Richards' equation}: application to multi-materials dam wetting},
    journal = {Advances in Water Resources},
    year = {2021},
    volume = {151},
    eid = {103897},
    doi = {10.1016/j.advwatres.2021.103897}
  }

@article{tian_h_adaptive_ldg_2027,
    author = {Tian, L. and Lu, L. and Guo, H. and Cheng, J. and Yang, Y.},
    title = {An {$h$}-adaptive local discontinuous {Galerkin} method for solving
          {Richards' equation} in unsaturated soils},
    journal = {Journal of Computational and Applied Mathematics},
    year = {2027},
    volume = {491},
    eid = {117953},
    doi = {10.1016/j.cam.2026.117953}
  }

@article{gassner_skew_symmetric_dg_2013,
    author = {Gassner, Gregor J.},
    title = {A skew-symmetric discontinuous {Galerkin} spectral element
        discretization and its relation to {SBP-SAT} finite difference
        methods},
    journal = {SIAM Journal on Scientific Computing},
    year = {2013},
    volume = {35},
    number = {3},
    pages = {A1233--A1253},
    doi = {10.1137/120890144}
  }

@article{manzanero_bassi_rebay_2018,
    author = {Manzanero, Juan and Rueda-Ram{\'i}rez, Andr{\'e}s M. and
        Rubio, Gonzalo and Ferrer, Esteban},
    title = {The {Bassi Rebay} 1 scheme is a special case of the symmetric
        interior penalty formulation for discontinuous {Galerkin}
        discretisations with {Gauss--Lobatto} points},
    journal = {Journal of Computational Physics},
    year = {2018},
    volume = {363},
    pages = {1--10},
    doi = {10.1016/j.jcp.2018.02.035}
  }

@article{clark_laugh_tests_2021,
    author = {Clark, Martyn P. and Zolfaghari, Reza and Green, Kevin R. and
        Trim, Sean and Knoben, Wouter J. M. and Bennett, Andrew and
        Nijssen, Bart and Ireson, Andrew and Spiteri, Raymond J.},
    title = {The numerical implementation of land models: Problem formulation
        and laugh tests},
    journal = {Journal of Hydrometeorology},
    year = {2021},
    volume = {22},
    number = {6},
    pages = {1627--1648},
    doi = {10.1175/JHM-D-20-0175.1}
  }

@article{ireson_openre_2023,
    author = {Ireson, Andrew M. and Spiteri, Raymond J. and Clark, Martyn P. and
        Mathias, Simon A.},
    title = {A simple, efficient, mass-conservative approach to solving
          {Richards' equation} ({openRE}, v1.0)},
    journal = {Geoscientific Model Development},
    year = {2023},
    volume = {16},
    number = {2},
    pages = {659--677},
    doi = {10.5194/gmd-16-659-2023}
  }

@article{ranocha_trixi_2022,
    author = {Ranocha, Hendrik and Schlottke-Lakemper, Michael and
        Winters, Andrew R. and Faulhaber, Erik and Chan, Jesse and
        Gassner, Gregor J.},
    title = {Adaptive numerical simulations with {Trixi.jl}: A case study of
          {Julia} for scientific computing},
    journal = {Proceedings of the JuliaCon Conferences},
    year = {2022},
    volume = {1},
    number = {1},
    eid = {77},
    doi = {10.21105/jcon.00077}
  }

@article{ranocha_error_based_step_size_2025,
    author = {Ranocha, Hendrik and Winters, Andrew R. and Castro, Hugo Guillermo
        and Dalcin, Lisandro and Schlottke-Lakemper, Michael and Gassner, Gregor J.
        and Parsani, Matteo},
    title = {On error-based step size control for discontinuous {Galerkin} methods
        for compressible fluid dynamics},
    journal = {Communications on Applied Mathematics and Computation},
    year = {2025},
    volume = {7},
    number = {1},
    pages = {3--39},
    doi = {10.1007/s42967-023-00264-y}
  }

@article{gustafsson_pi_stepsize_1988,
    author = {Gustafsson, Kjell and Lundh, Michael and S{\"{o}}derlind, Gustaf},
    title = {A {PI} stepsize control for the numerical solution of ordinary
        differential equations},
    journal = {BIT Numerical Mathematics},
    year = {1988},
    volume = {28},
    number = {2},
    pages = {270--287},
    doi = {10.1007/BF01934091}
  }

@article{steinebach_rodas5p_2023,
    author = {Steinebach, Gerd},
    title = {Construction of {Rosenbrock--Wanner} method {Rodas5P} and numerical
        benchmarks within the {Julia} {Differential Equations} package},
    journal = {BIT Numerical Mathematics},
    year = {2023},
    volume = {63},
    number = {2},
    eid = {27},
    doi = {10.1007/s10543-023-00967-x}
  }

@article{rackauckas_differentialequations_2017,
    author = {Rackauckas, Christopher and Nie, Qing},
    title = {{DifferentialEquations.jl} -- A performant and feature-rich ecosystem
        for solving differential equations in {Julia}},
    journal = {Journal of Open Research Software},
    year = {2017},
    volume = {5},
    number = {1},
    eid = {15},
    doi = {10.5334/jors.151}
  }

@misc{revels_forward_mode_automatic_differentiation_2016,
    author = {Revels, Jarrett and Lubin, Miles and Papamarkou, Theodore},
    title = {Forward-mode automatic differentiation in {Julia}},
    year = {2016},
    eprint = {1607.07892},
    eprintclass = {cs.MS}
  }

@misc{montoison_revisiting_sparse_matrix_coloring_2025,
    author = {Montoison, Alexis and Dalle, Guillaume and Gebremedhin, Assefaw},
    title = {Revisiting sparse matrix coloring and bicoloring},
    year = {2025},
    eprint = {2505.07308},
    eprintclass = {math.NA}
  }

@article{davis_klu_2010,
    author = {Davis, Timothy A. and Palamadai Natarajan, Ekanathan},
    title = {Algorithm 907: {KLU}, a direct sparse solver for circuit simulation
        problems},
    journal = {ACM Transactions on Mathematical Software},
    year = {2010},
    volume = {37},
    number = {3},
    pages = {36:1--36:17},
    doi = {10.1145/1824801.1824814}
  }

@article{haverkamp_comparison_infiltration_1977,
    author = {Haverkamp, R. and Vauclin, M. and Touma, J. and
        Wierenga, P. J. and Vachaud, G.},
    title = {A comparison of numerical simulation models for one-dimensional
        infiltration},
    journal = {Soil Science Society of America Journal},
    year = {1977},
    volume = {41},
    number = {2},
    pages = {285--294},
    doi = {10.2136/sssaj1977.03615995004100020024x}
  }

@article{van_genuchten_closed_form_1980,
    author = {van Genuchten, M. Th.},
    title = {A closed-form equation for predicting the hydraulic conductivity of
        unsaturated soils},
    journal = {Soil Science Society of America Journal},
    year = {1980},
    volume = {44},
    number = {5},
    pages = {892--898},
    doi = {10.2136/sssaj1980.03615995004400050002x}
  }

@article{mualem_new_model_1976,
    author = {Mualem, Yechezkel},
    title = {A new model for predicting the hydraulic conductivity of unsaturated
        porous media},
    journal = {Water Resources Research},
    year = {1976},
    volume = {12},
    number = {3},
    pages = {513--522},
    doi = {10.1029/WR012i003p00513}
  }

@article{keita_time_stepping_richards_2021,
    author = {Keita, Sana and Beljadid, Abdelaziz and Bourgault, Yves},
    title = {Implicit and semi-implicit second-order time stepping methods for
        the {Richards} equation},
    journal = {Advances in Water Resources},
    year = {2021},
    volume = {148},
    eid = {103841},
    doi = {10.1016/j.advwatres.2020.103841}
  }

@article{kees_higher_order_time_integration_2002,
    author = {Kees, C. E. and Miller, C. T.},
    title = {Higher order time integration methods for two-phase flow},
    journal = {Advances in Water Resources},
    year = {2002},
    volume = {25},
    number = {2},
    pages = {159--177},
    doi = {10.1016/S0309-1708(01)00054-9}
  }

@article{cumming_mass_conservative_2011,
    author = {Cumming, Ben and Moroney, Tim and Turner, Ian},
    title = {A mass-conservative control volume-finite element method for solving
          {Richards'} equation in heterogeneous porous media},
    journal = {BIT Numerical Mathematics},
    year = {2011},
    volume = {51},
    number = {4},
    pages = {845--864},
    doi = {10.1007/s10543-011-0335-3}
  }

@article{bootsma_modified_picard_2026,
    author = {Bootsma, Huite P. and van der Ploeg, Martine J. and
        Weerts, Albrecht H.},
    title = {Why modified {Picard} works: Mass-conservative and higher-order time
        integration for the {Richardson--Richards} equation},
    journal = {Advances in Water Resources},
    year = {2026},
    volume = {213},
    eid = {105327},
    doi = {10.1016/j.advwatres.2026.105327}
  }

@article{farthing_richards_review_2017,
    author = {Farthing, Matthew W. and Ogden, Fred L.},
    title = {Numerical solution of {Richards' equation}: A review of advances
        and challenges},
    journal = {Soil Science Society of America Journal},
    year = {2017},
    volume = {81},
    number = {6},
    pages = {1257--1269},
    doi = {10.2136/sssaj2017.02.0058}
  }

@article{zha_review_numerical_solution_2019,
    author = {Zha, Yuanyuan and Yang, Jinzhong and Zeng, Jicai and
        Tso, Chak-Hau M. and Zeng, Wenzhi and Shi, Liangsheng},
    title = {Review of numerical solution of {Richardson--Richards} equation
        for variably saturated flow in soils},
    journal = {WIREs Water},
    year = {2019},
    volume = {6},
    number = {5},
    eid = {e1364},
    doi = {10.1002/wat2.1364}
  }

@article{chatzikamaris_numerical_solvers_2026,
    author = {Chatzikamaris, Athanasios and Bourantas, Georgios and
        Burganos, Vasilis N. and Malamos, Nikolaos},
    title = {Numerical solvers for the {Richards} equation: A comparative
        review of stability, efficiency and mass conservation},
    journal = {Eng},
    year = {2026},
    volume = {7},
    number = {9},
    eid = {467},
    doi = {10.3390/eng7090467}
  }

@book{hairer_solving_1996,
    author = {Hairer, Ernst and Wanner, Gerhard},
    title = {Solving ordinary differential equations {II}: Stiff and
        differential-algebraic problems},
    edition = {2nd},
    publisher = {Springer},
    year = {1996},
    doi = {10.1007/978-3-642-05221-7}
  }

@book{hundsdorfer_numerical_2003,
    author = {Hundsdorfer, Willem and Verwer, Jan G.},
    title = {Numerical solution of time-dependent
        advection-diffusion-reaction equations},
    series = {Springer Series in Computational Mathematics},
    volume = {33},
    publisher = {Springer},
    year = {2003},
    doi = {10.1007/978-3-662-09017-6}
  }

@book{kopriva_implementing_spectral_methods_2009,
    author = {Kopriva, David A.},
    title = {Implementing spectral methods for partial differential equations:
        Algorithms for scientists and engineers},
    publisher = {Springer},
    year = {2009},
    doi = {10.1007/978-90-481-2261-5}
  }

@book{stewart_sun_matrix_perturbation_theory_1990,
    author = {Stewart, G. W. and Sun, Ji-guang},
    title = {Matrix perturbation theory},
    publisher = {Academic Press},
    year = {1990},
    isbn = {978-0-12-670230-9}
  }

@misc{montoya_spiteri_adaptive_richards_reproducibility_2026,
    author = {Montoya, Tristan and Spiteri, Raymond J.},
    title = {Reproducibility Repository for ``{A} Fully Discrete,
        Mass-Conservative Adaptive Discontinuous Spectral-Element
        Method for the One-Dimensional {Richards} Equation in Mixed Form''},
    year = {2026},
    howpublished = {Zenodo},
    note = {Version 0.1.2},
    doi = {10.5281/zenodo.23270657}
  }

\end{document}